\documentclass[12pt,reqno]{amsart}

\usepackage{lipsum}

\usepackage{tikz}
\usetikzlibrary{decorations.pathreplacing}
\usepackage{amsfonts,latexsym,rawfonts,amsmath,amssymb,amsthm, mathrsfs, lscape}

\usepackage[all]{xy}
\usepackage{graphicx}
\usepackage{dsfont}
\usepackage{bm}

\usepackage{array, tabularx}

\usepackage{setspace}
\usepackage[english]{babel}
\usepackage{color}
\usepackage{enumerate}

\usepackage{hyperref}

\usepackage{anysize}
\marginsize{2.5cm}{2.5cm}{2cm}{2.5cm}

\newtheorem{theorem}{Theorem}[section]

\newtheorem{corollary}{Corollary}[theorem]
\newtheorem{definition}[theorem]{Definition}

\newtheorem{lemma}[theorem]{Lemma}

\newtheorem{proposition}[theorem]{Proposition}
\newtheorem{remark}{Remark}[theorem]

\numberwithin{equation}{section}
\numberwithin{figure}{section}

\newcommand{\bp}{\bar{\partial}}

\newcommand{\bx}{\bm{x}}
\newcommand{\by}{\bm{y}}

\newcommand{\Ca}{\mathcal{C}^n}

\newcommand{\dC}{\mathbb{C}}

\newcommand{\dN}{\mathbb{N}}
\newcommand{\dR}{\mathbb{R}}

\newcommand{\dZ}{\mathbb{Z}}

\newcommand{\hf}{\widehat{F}}
\newcommand{\hg}{\widehat{G}}
\newcommand{\hq}{\widehat{Q}}

\newcommand{\ldel}{\underline{\delta}}

\newcommand{\p}{\partial}
\newcommand{\R}{\mathbb R}

\newcommand{\ty}{\epsilon_X}
\newcommand{\tlk}{\tilde{\lambda}_k}

\newcommand{\vf}{\varphi}

\DeclareMathOperator{\dvol}{dvol}

\DeclareMathOperator{\fa}{\alpha}
\DeclareMathOperator{\fb}{\beta}

\DeclareMathOperator{\Ric}{Ric}

\DeclareMathOperator{\Ku}{\Phi^{\sharp}}
\DeclareMathOperator{\Tr}{Tr}
\DeclareMathOperator{\Tri}{\Psi^{\flat}}
\DeclareMathOperator{\Vol}{Vol}

\begin{document}

\title{A Liouville theorem on asymptotically Calabi spaces}

\author{Song Sun}

\address{Department of Mathematics, University of California, Berkeley, CA, USA, 94720} 
\email{sosun@berkeley.edu}

\author{Ruobing Zhang}
 \thanks{The first author is supported by NSF Grant DMS-1708420, an Alfred P. Sloan Fellowship, and the Simons Collaboration Grant on Special Holonomy in Geometry, Analysis and Physics ($\#$ 488633, S.S.). The second author is supported by  NSF Grant DMS-1906265. }

\address{Department of Mathematics, Princeton University, Princeton, NJ, USA, 08544}
\email{ruobingz@princeton.edu}

\begin{abstract}

In this paper, 
we will study harmonic functions on 
the complete and incomplete spaces with nonnegative Ricci curvature which exhibit inhomogeneous collapsing behaviors at infinity.
The main result states that any nonconstant harmonic function on such spaces yields a definite  exponential growth rate which depends explicitly on the geometric data at infinity.

\end{abstract}

\maketitle

\tableofcontents

\section{Introduction}
Our main goal in this paper is to prove Liouville type theorems for harmonic functions on a class of  non-compact Riemannian manifolds exhibiting inhomogeneous collapsing behaviors at infinity, in both of the complete and incomplete settings. These results provide crucial technical tools to  the weighted analysis in \cite{SZ}.

\subsection{Background}

To begin with, we will briefly introduce the motivation of studying the harmonic functions on such spaces. In the authors' recent program of studying the complex structure degenerations and  collapsing of Calabi-Yau manifolds (see \cite{SZ}), the main ingredient is to construct the collapsing Calabi-Yau metrics and accurately describing the singularity behaviors when the complex structures are degenerating.  To achieve this, the  weighted analysis compatible with the singularity behaviors of the degenerating family was developed in \cite{SZ}. A necessary technical step in establishing the uniform weighted estimates 
is to prove the Liouville type theorems regarding the linearized operators on numerous rescaling bubble limits. In this paper, our goal is to prove the Liouville type theorems on those most complicated rescaling bubbles appearing in the weighted analysis in \cite{SZ}. Since both complete and incomplete Calabi-Yau manifolds were studied in \cite{SZ}, in addition to prove the Liouville type theorem on a complete non-compact manifold, we will also formulate the Liouville theorems employed in the weighted analysis in \cite{SZ} with the appropriate boundary conditions.

As a preliminary, let us recall the definition of a \emph{Calabi model space}, as in \cite{HSVZ}. Let $D$ be a compact complex manifold of complex dimension $n-1$ with a nowhere vanishing holomorphic volume form $\Omega_D$ and let $\omega_D$ be a Calabi-Yau metric in the K\"ahler class $2\pi c_1(L)$ for an ample line bundle $L$. We also fix a hermitian metric $h$ on $L$ with curvature form $-\sqrt{-1}\omega_D$. The Calabi model space $\Ca$ is the subset of the total space of $L$ consisting of elements $\xi$ with $0<|\xi|_h<1$, which is equipped with a holomorphic volume form $\Omega_{\Ca}$ and an incomplete Calabi-Yau metric $\omega_{\Ca}$. For our purpose in this article the holomorphic volume form $\Omega_{\Ca}$ does not play a role so we omit its formula. The K\"ahler form $\omega_{\Ca}$ is given by the {\it Calabi ansatz} and written as 
 \begin{equation}\omega_{\Ca}=\frac{n}{n+1}\sqrt{-1}\p\bp (-\log |\xi|_h^2)^{\frac{n+1}{n}}.
 \end{equation}
 The corresponding Riemannian metric $g_{\Ca}$ is Ricci flat, which is incomplete as $|\xi|_h\rightarrow 1$ and complete as $|\xi|_h\rightarrow 0$. In the complete end, the metric $g_{\Ca}$  exhibits  non-standard geometric behavior, which is described as follows. There is a natural $S^1$-action on $\Ca$ given by fiberwise rotation, and the corresponding moment map is given by
 \begin{equation}
 z=(-\log|\xi|_h^2)^{1/n}.\label{e:moment-map-definition}
 \end{equation}
 The relationship between $z$ and the distance function $r$ to a fixed point is given by 
 \begin{equation}
 C^{-1} z^{\frac{n+1}{2}}\leq r\leq C z^{\frac{n+1}{2}},  \ \ z\geq 1. 
 \end{equation}
So the model space $\Ca$ is naturally  diffeomorphic to a topological product $\R^+\times Y^{2n-1}$, where the compact fiber $Y^{2n-1}$ is a circle bundle 
\begin{equation}
S^1 \to Y^{2n-1} \to D
\end{equation}
based over $D$, and $z$ is the coordinate on $\R_+$.  As $r\rightarrow \infty$, the length of the $S^1$-orbits has size comparable with $r^{\frac{1-n}{n+1}}$ while the diameter of the base $D$ is comparable with $r^{\frac{1}{n+1}}$. In addition, as $r\to \infty$, the volume growth has the following fractional rate
\begin{equation}
\Vol_{g_{\Ca}}(B_r(p)) \sim r^{\frac{2n}{n+1}}
\end{equation}
for any fixed reference point $p$, and the tangent cone at infinity is isometric to the half line $\dR_+$.

\begin{definition}[$\delta$-aysmptotically Calabi space]\label{d:asymptotic-Calabi}
Given some constant $\delta>0$,
a complete Riemannian manifold $(X^{2n}, g)$ of dimension $2n$ is said to be $\delta$-\emph{asymptotically Calabi}
 if there exist a compact subset $K\subset X^{2n}$, a Calabi model space $(\Ca, g_{\Ca})$ and a diffeomorphism 
 \begin{equation}\Phi: \Ca \setminus K'\rightarrow X^{2n}\setminus K\end{equation} 
with $K' = \{|\xi| \geq C\} \subset \Ca$ (for some $C>0$) such that for all $k \in \dN$,
\begin{equation}|\nabla_{g_{\Ca}}^k(\Phi^*g-g_{\Ca})|_{g_{\mathcal C^n}}=O(e^{-\delta z^{\frac{n}{2}}}) \ \text{as} \ z \to +\infty.
\end{equation}
Figure \ref{f:TY} describes the asymptotic behavior of a $\delta$-asymptotically Calabi space $X^{2n}$.

\end{definition}

\begin{figure}
\begin{tikzpicture}[scale = 1.2]
\draw (0,0) to [out = 90, in = 180] (1,.5) to (2,.5);

\draw (2,.5) to [out = 0, in = 235] (4,2);
\draw (4,2) to [out = 55, in = 265] (4.6,3.3);

\draw (0,0) to [out = 270, in = 180] (1,-.5) to (2,-.5);
\draw (2,-.5) to [out = 0, in = 125] (4,-2);
\draw (4,-2) to [out = -55, in = -265] (4.6,-3.3);

\draw[blue](2,0) ellipse (.2 and .5);
\draw[red](2,0) ellipse (.2 and .05);

\draw[blue] (3,0) ellipse (.1 and .82);

\draw[red](3,0) ellipse (.1 and .05);
\draw[blue](4,0) ellipse (.05 and 2);

\draw[red ](4,0) ellipse (.05 and .02);

\draw (.8,-.05) arc [ radius = .5, start angle = 45, end angle = 135];
\draw (1.2,-.05) arc [ radius = .5, start angle = 45, end angle = 135];
\draw (1.6,-.05) arc [ radius = .5, start angle = 45, end angle = 135];

\node at (0,-1) {$X^{2n}$};

\node at (4.5,0) {$\Ca$};

\end{tikzpicture}
\caption{The Calabi model $\Ca$ appears near infinity of $X^{2n}$: the red circles are the $S^1$-fibers, while the blue curves represent the divisor $D$.} \label{f:TY}
\end{figure}
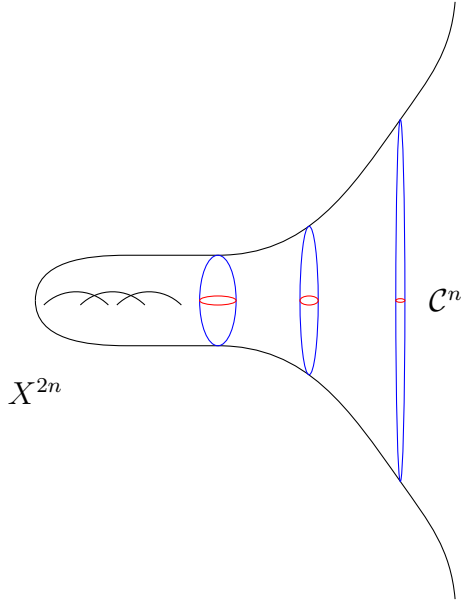

\subsection{Main results}

The first main result of the paper is the following Liouville type theorem for the harmonic functions on a complete $\delta$-asymptotically Calabi space. 
\begin{theorem}
\label{t:liouville-theorem-calabi}
Let $(X^{2n}, g)$ be a $\delta$-asymptotically Calabi space for some $\delta > 0$ and with non-negative Ricci curvature. Let $(\Ca, g_{\Ca})$
be the Calabi model space of $X^{2n}$ based over a compact Calabi-Yau manifold $(D,\omega_D)$. Denote by 
\begin{align}
\ty \equiv \min\left\{\delta, 2\cdot\Big(\frac{\lambda_D}{n}\Big)^{\frac{1}{2}}\right\},
\end{align}
where $\lambda_D>0$ is the smallest positive eigenvalue of $-\Delta_D$.
If $u$ is a harmonic function on $(X^{2n}, g)$ which satisfies the growth condition \begin{equation}|u| = O(e^{ \epsilon \cdot z^{\frac{n}{2}}}), \ z\to+\infty,\end{equation} 
for some $\epsilon\in(0,\ty)$, then $u$ is a constant.
\end{theorem}

\begin{remark}
The main application of this theorem is in \cite{SZ}, where $(X^{2n},g)$ is a complete Tian-Yau space constructed in \cite{TY}. The underlying complex manifold is the complement of a smooth anti-canonical divisor in a closed Fano manifold. It was proved in \cite{HSVZ} that a Tian-Yau space is always $\delta$-asymptotically Calabi for some $\delta>0$.
 \end{remark}

\begin{remark}The special case $n=2$ of Theorem \ref{t:liouville-theorem-calabi} was proved in \cite{HSVZ}.
\end{remark}

\begin{remark}
Although not needed in \cite{SZ}, it is interesting to see if there is a general Fredholm theory for the analysis of the Laplace  operator on such spaces. We asked similar questions in the two dimensional case in \cite{HSVZ}. 
\end{remark}

\begin{remark}
Notice that the asymptotic cone 
of a  $\delta$-asymptotically Calabi space $(X^{2n}, g)$ is a half line. We are informed by Gilles Carron that, in the case of $1$-dimensional asymptotic cone, one can make use of the tools developed in \cite{Carron} to understand the harmonic functions on such complete spaces. Specifically,
by slightly modifying the proof of theorem 2.3 in \cite{Carron}, we can conclude that if a harmonic function $u$ on $(X^{2n},g)$ satisfies $u=O(e^{\delta \cdot z^{\frac{n}{2}}})$ for some $\delta>0$, then $u$ must be a constant. It is worth mentioning that the exponential gap obtained in \cite{Carron} is implicit, while we obtain an explicit and sharp gap in Theorem \ref{t:liouville-theorem-calabi}. This explicit gap is helpful to develop the Fredholm theory on such complete spaces.	
\end{remark}

The next result gives a Liouville type theorem for the solutions of the Neumann boundary problem on the incomplete Calabi model space.

\begin{theorem}[c.f. Corollary \ref{c:Liouville-neumann}]\label{t:Liouville-Calabi-neumann}Let $(\Ca, g_{\Ca})$ be an incomplete Calabi model space  based over a compact Calabi-Yau manifold $(D,\omega_D)$ and the natural moment map coordinate $z$ as in \eqref{e:moment-map-definition}, so that $\Ca$ is diffeomorphic to a topological product $[z_0,+\infty)\times Y^{2n-1}$ for some $z_0>0$ under the moment map coordinate $z$, where $Y^{2n-1}$ is a circle bundle over $D$. 
 Let $u$ be a solution of the Neumann boundary problem \begin{align}
\begin{cases}
\Delta_{g_{\Ca}} u(\bx) = 0, & \bx\in\Ca,
\\
\frac{\p u}{\p z}(\bx) = \kappa_0, & z(\bx)= z_0,
\end{cases}
\end{align}
where $\kappa_0\in\dR$.
If $u$ satisfies the growth condition $|u(\bx)| = O(e^{\delta\cdot z(\bx)^{\frac{n}{2}}})$ for some  $0<\delta<2(\frac{\lambda_D}{n})^{\frac{1}{2}}$, 
then there is some $\ell_0\in\dR$ such that $u=\kappa_0\cdot z + \ell_0$ on $\Ca$.
In particular, $u$ is constant on $\Ca$ when $\kappa_0=0$.
\end{theorem}

As a comparison, we obtain the following Liouville theorem under the Dirichlet boundary condition on the incomplete Calabi model space.

\begin{theorem}
[c.f. Corollary \ref{c:Liouville-dirichlet}]\label{t:Liouville-Calabi-Dirichlet} In the above notations, let $u$ be a solution of the Dirichlet boundary problem on the Calabi space $(\Ca, g_{\Ca})$ based over a compact Calabi-Yau manifold $(D,\omega_D)$,
\begin{align}
\begin{cases}
\Delta_{g_{\Ca}} u(\bx) = 0, & \bx\in\Ca,
\\
u(\bx) = 0, & z(\bx)= z_0.
\end{cases}
\end{align}
If $u$ satisfies the growth condition $|u(\bx)| = O(e^{\delta\cdot z(\bx)^{\frac{n}{2}}})$ for some  $0<\delta<2(\frac{\lambda_D}{n})^{\frac{1}{2}}$, 
then $u$ must be a linear function on $\Ca$, i.e., there exists a constant $\kappa_0\in\dR$ such that  $u = \kappa_0 \cdot (z-z_0)$ in terms of the natural moment map coordinate $z$ on $\Ca$.

\end{theorem}

The proof of the above theorems rely on the technique of separation of variables and careful analysis on the special functions. The treatments in this paper are fairly elementary and transparent so that the general strategy are expected to be applied in more general contexts. In addition to the direct applications to the weighted and bubbling analysis in the degenerations problems (e.g. \cite{SZ} and \cite{HSVZ}), we are expecting the further developments involving the sharp Fredholm theory on the spaces with nonstandard collapsing behaviors at infinity.
An example of interesting directions is the delicate analysis of the moduli space structure based on a sharp Fredholm theory.
 Our work can be viewed as an initial step towards this direction.

The paper is organized as follows. In Section \ref{s:ode-setup}, we will recall the separation of variables in \cite{HSVZ} and write down the ODE for the Laplace equation on the Calabi model space. This ODE is not familiar at first sight, which leads us to  perform the change of variables to transform the ODE to known  ones. Depending on whether the Fourier mode with respect to the natural $S^1$-action vanishes or not, we shall get either \emph{modified Bessel equations}, or   \emph{confluent hypergeometric equations}. Solutions to these equations have known asymptotics, but for our analysis we need {\it uniform estimates}, which will be established in Section \ref{s:j=0} and \ref{s:jk not zero}. The key technical ingredients involve estimating exponential integrals using {\it Laplace's method}. To make the paper self-contained, in Appendix \ref{s:appendix-1},
 we will summarize some known facts and technical integration formulas for special functions. 
 With these preparations, in Section \ref{s:asymp harmonic Calabi model}, we  prove that a harmonic function on the Calabi model space which has slowly exponential growth at infinity must decompose as the sum of the linear function in $z$ (the moment coordinate in the Calabi model space)  and exponentially decaying terms.  We will also study the Liouville theorems for boundary problems on the incomplete Calabi space by proving Theorems \ref{t:Liouville-Calabi-neumann} and \ref{t:Liouville-Calabi-Dirichlet}. In Section \ref{s:Poisson with asymptotics}, we show that the Poisson equation on the Calabi model space can be solved using separation of variables for a function with certain growth control at infinity. Section \ref{s:proof of Liouville} is dedicated to the proof of  Theorem \ref{t:liouville-theorem-calabi}.  First transplanting the harmonic function to an approximately harmonic function on the Calabi model space, then correct this to a harmonic function by solving a Poisson equation. These imply the function $u$ grows at most linearly in $z$. The later then implies $du$ is a decaying harmonic 1-form, and must vanish by applying the Bochner technique (which uses the assumption $\Ric_g\geq 0$) and maximum principle.

Now we list some notations and make basic conventions for the convenience of later discussions in this section:
\begin{itemize}
\item The Laplace-Beltrami operator $\Delta_g$ acting on functions is given by
\begin{equation}
\Delta_g u \equiv \Tr_g(\nabla^2u).
\end{equation}
For example, $\Delta_{\dR^n}\equiv \sum\limits_{j=1}^n\frac{\partial^2}{\partial x_j^2}$.

\item Let $k\in\dZ_+$ and $x\in\dR$, we define \begin{equation}(x)_k\equiv \prod\limits_{m=1}^{k}(x+m-1)\ and \ (x)_0\equiv 1.\end{equation} 

\item Given two positive functions $f(z)$ and $g(z)$ defined on $\dR_+$, then 
\begin{enumerate}

\item We say $f(z)\sim g(z)$ if 
\begin{equation}
\lim\limits_{z\to+\infty}\frac{f(z)}{g(z)}=1.
\end{equation}
\item 

Given two $C^1$-functions $f(z)$ and $g(z)$,
then their \emph{Wronskian} is denoted by 
\begin{equation}\mathcal{W}(f,g)(z)\equiv f(z)g'(z)-f'(z)g(z).\end{equation}

\end{enumerate}
\end{itemize}

\subsection{Acknowledgments}

We thank Gilles Carron for bringing to our attention the paper \cite{Carron} after this paper was submitted  which gives another approach to understanding the harmonic functions on the spaces with $1$-dimensional asymptotic cones. We are also grateful to the anonymous referees whose suggestions substantially improved the presentation of the paper.

\section{Separation of variables and ODE reduction}

\label{s:ode-setup}

The general ideas of performing separation of variables on a complete space with certain model geometry at infinity have been extensively explored in the literature such as  asymptotically cylindrical spaces  and asymptotically conical spaces. An example of earlier works in this direction is \cite{LM}, where the elliptic operators were carefully studied in the asymptotically cylindrical cases.
It is different from the standard model geometry (e.g. a cone or a cylinder) that
 the complicated collapsing behaviors at the infinity of the Calabi model space lead to involved separation of variables.  
  In order to carry out separation of variables, we  will study the local representation of the Laplace operator $\Delta_{\Ca}$ on $\Ca$. 
The separation of variables has been developed in Section 4.1 of \cite{HSVZ}, and here we just briefly review the computations and basic estimates.

Let $\{z_i\}_{i=1}^{n-1}$ be some local holomorphic coordinates on  $D$, and fix a local holomorphic trivialization  $e_0$ of the line bundle $L$  with $|e_0|^2=e^{-\psi}$,  where $\psi:D \to \dR$ is a smooth function. So we get local holomorphic coordinates $(\underline{z},\zeta)\equiv(z_1,\ldots,z_{n-1}, w)$ on $\Ca$ by writing a point  $\xi\in \Ca$ as $\xi=w \cdot e_0(\underline{z})$. Then $|\xi|_h^2=|w|^2e^{-\psi}$, where we may assume 
\begin{align}\psi(0)=1, \ d\psi(0)=0,\ \sqrt{-1}\partial\bar\partial\psi=\omega_D.\end{align} 
Let $\pi:\Ca\rightarrow D$ be the natural bundle projection map. We denote 
\begin{equation}\varrho\equiv|\xi|_h,
\end{equation}
then  \begin{equation}w=\varrho e^{\frac{\psi}{2}+\sqrt{-1}\theta},\end{equation} where $\p_\theta$ generates the natural $S^1$-rotation on the total space of $ L$. The K\"ahler form $\omega_{\Ca}$ of the Calabi model space can be written as \begin{equation}
 \label{Calabiformula}
 \omega_{\Ca}=(-\log |\xi|^2_h)^{\frac{1}{n}}\omega_D+\frac{1}{n}(-\log |\xi|_h^2)^{\frac{1}{n}-1} \sqrt{-1}(\frac{dw}{w}-\p\psi)\wedge (\frac{d\bar w}{\bar w}-\bp\psi).
 \end{equation}

Now we fix some $r_0\in (0, 1)$, and define $(Y^{2n-1}, h_0)$ to be the level set $\{\varrho=r_0\}$ endowed with the induced Riemannian metric $h_0$, which has an explicit representation
\begin{align}
h_0=(-\log r_0^2)^{\frac{1}{n}} g_D+\frac{1}{n} (-\log r_0^2)^{\frac{1}{n}-1}(d\theta-\frac{1}{2}d^c\psi)\otimes (d\theta-\frac{1}{2}d^c\psi).\label{e:h_0}
\end{align}
 We denote by $\{\Lambda_{k}\}_{k=0}^{\infty}$ 
the spectrum of $\Delta_{h_0}$ with $\Lambda_0\equiv0$, and let $\{\vf_k\}_{k=0}^{\infty}$ be an orthonormal basis of (complex-valued) eigenfunctions which are homogeneous under the $S^1$-action and with
\begin{equation}
-\Delta_{h_0} \vf_k = \Lambda_k \cdot \vf_k.
\end{equation}
From \cite{HSVZ}, Section 4.1 we know that $\Lambda_k$ can be always represented as follows, \begin{equation} \label{e:definition of Lambdak}
\Lambda_k = \frac{\lambda_k}{z_0} + nz_0^{n-1} \cdot j_k^2
\end{equation}
such that $j_k\in\dN$ and \begin{equation}\lambda_k \geq \frac{(n-1)\cdot j_k}{2}.\label{e:lambda_k-lower-bound}\end{equation}
Notice $j_k$ and $\lambda_k$ have geometric meanings as explained in \cite{HSVZ}, Section 4.1. Namely, $\vf_k$ has weight $\pm j_k$ with respect to the $S^1$-action (notice the weight of $\bar\vf_k$ is negative the weight of $\vf_k$), and $\vf_k$ corresponds to a smooth section of $L^{-j_k}$ over $D$, which is an eigenfunction of the $\bp$-Hodge Laplacian  with eigenvalue $\lambda_k$. In particular $\lambda_0=j_0=0$ and $\vf_0$ is a constant. Moreover when $j_k=0$, $\vf_k$ corresponds to an eigenfunction on $D$ and 
\begin{equation} \label{e:definition of underline lambda}
\lambda_D\equiv \inf \{\lambda_k>0|j_k=0, k\in\dZ_+\}>0.
\end{equation}

Now we carry out separation of variables for the Laplace operator on $\Delta_{\Ca}$.
Let $u$ be a harmonic function on the model space $\Ca$, namely,
\begin{equation}
\Delta_{\Ca} u = 0.
\end{equation}
In the following, for any $\xi\in\Ca$, we will denote by $z$ the natural moment map coordinate as in \eqref{e:moment-map-definition}.  
For every fixed $z$, we can write the $L^2$-expansion along the fiber $Y^{2n-1}$,
\begin{equation}
u(z, \by) = \sum\limits_{k=1}^{\infty} u_k(z) \cdot \vf_k(\by).\label{e:sum-up-harmonic}
\end{equation} For completeness and reader's convenience, let us recall the separation of variables described in \cite{HSVZ}.  Also we notice that  
\begin{align}
\frac{\p \varrho}{\p w} & =\frac{\varrho}{2w},\quad  \frac{\p \varrho}{\p z_i}=-\frac{1}{2}\varrho\cdot\p_{z_i}\psi,
\\
\frac{\p\theta}{\p w} & =\frac{1}{2\sqrt{-1}w},\quad   \frac{\p \theta}{\p z_i}=0 , \end{align}
and 
\begin{equation}
|w|^2\frac{\p^2 u}{\p w\p \bar w} =\frac{1}{4}(\varrho^2u_{\varrho\varrho}+\varrho u_{\varrho}+u_{\theta\theta}).
\end{equation}
 Now the Laplacian at  points in the fiber $\pi^{-1}(0)$ is given by 
\begin{equation}\Delta_{\Ca}u= (-\log |\xi|^2_h)^{-\frac{1}{n}} \sum_{i=1}^{n-1}\frac{\partial^2u}{\partial z_i \partial \bar{z}_i} +n(-\log |\xi|_h^2)^{-\frac{1}{n}+1} |w|^2 \frac{\partial^2 u}{\partial w\partial \bar w}.\end{equation}
Let us consider a smooth function $\phi\in C^{\infty}(Y^{2n-1})$ with \begin{equation}\mathcal L_{\p_\theta}\phi=\sqrt{-1}\cdot j\cdot \phi\label{e:S1-action}\end{equation} for some integer $j\in\dZ$. Replacing $\phi$ by $\bar\phi$ if necessary we may assume $j\geq 0$. 
Following the same computations as in \cite{HSVZ}, 
for  a smooth function $u(\varrho, z)= f(\varrho) \phi(y)$ and re-label $u$ by $u_k$ as in \eqref{e:definition of Lambdak}, so
$u_k(z)$ satisfies the differential equation
\begin{equation}\frac{d^2 u_k(z)}{dz^2}-(\frac{j_k^2n^2}{4}\cdot z^{n}+n\lambda_k)z^{n-2}u_k(z)=0,\ z\geq 1.\label{e:homogeneous-ODE}\end{equation}
 We also consider the Poisson equation 
\begin{equation}
\Delta_{\Ca} u = v.
\end{equation}
Take the $L^2$-expansion of $v$ in the direction of the cross section $Y^{2n-1}$,
\begin{equation}
v(z,\by) = \sum\limits_{k=1}^{\infty}\xi_k(z) \cdot \vf_k(\by),\label{e:sum-up-poisson}
\end{equation}
then the same procedure of separation of variables leads to an ordinary differential equation 
\begin{equation}\frac{d^2 u_k(z)}{dz^2}-(\frac{j_k^2n^2}{4}\cdot z^{n}+n\lambda_k)z^{n-2}u_k(z)=z^{n-1}\cdot \xi_k(z), \ z\geq 1.\label{e:non-homogeneous}\end{equation}

Since we will study the solutions \eqref{e:sum-up-harmonic} 
and \eqref{e:sum-up-poisson} in terms of the fiber-wise $L^2$-expansions, so there are two fundamental ingredients to analyze: First, in order to show the $L^2$-expansions in fact converge, we need to obtain some {\it uniform estimates} for the ODE solutions 
which are independent of the subscript $k\in\dN$.
The other basic aspect is to understand the asymptotics of the linearly independent solutions $\mathcal{G}_k(z)$ and $\mathcal{D}_k(z)$ as $z\to+\infty$, which in turn gives the asymptotics of the solutions \eqref{e:sum-up-harmonic} and \eqref{e:sum-up-poisson}.

Notice that the cross section $Y^{2n-1}$ is a circle bundle over the divisor $D$,  then there are two different modes depending upon if the eigenfunctions $\varphi_k$ of $\Delta_{Y^{2n-1}}$ is $S^1$-invariant. By \eqref{e:S1-action}, the circle action 
\begin{equation}\mathcal L_{\p_\theta}\phi=\sqrt{-1}\cdot j\cdot \phi \end{equation}
is trivial if and only if $j=0$. 
More technically speaking, we will study the solutions to \eqref{e:homogeneous-ODE} and \eqref{e:non-homogeneous} in two different cases: $j_k = 0$ and $j_k \neq 0$. The first step is to understand the solutions to homogeneous equation \eqref{e:homogeneous-ODE}.
Notice that, by using the change of variables $\zeta\equiv z^n$, \eqref{e:homogeneous-ODE} will become a homogeneous equation with linear coefficients, so that we can apply the theory of special functions to obtain some effective estimates for the solutions. Now letting
 \begin{align}
 \begin{cases}
 \zeta = -\log r^2 =z^n
 \\
 w_k(\zeta) \equiv u_k(z) =  u_k(\zeta^{\frac{1}{n}}), 
 \end{cases}
 \label{e:convert-to-linear}
 \end{align}
 we have
 \begin{equation}
\zeta \cdot \frac{d^2w_k(\zeta)}{d\zeta^2} + (1-\frac{1}{n})\frac{dw_k(\zeta)}{d\zeta} - (\frac{j_k^2}{4}\cdot \zeta + \frac{\lambda_k}{n}) w_k(\zeta) = 0.\label{e:linear-coefficient-ode}
\end{equation}

In the first case  $j_k=0$, we make the transformation of the above solution $w(\zeta)$ as follows,
\begin{equation}
\begin{cases}
y=2\sqrt{\frac{\lambda}{n}}\cdot  \zeta^{\frac{1}{2}} \geq 0
\\
w_k(\zeta)= \zeta^{\frac{1}{2n}}\cdot \mathcal{B}\Big(2\sqrt{\frac{\lambda}{n}}\cdot  \zeta^{\frac{1}{2}}\Big),
\end{cases}
 \label{e:homogeneous-transformation}
\end{equation}
then the function $\mathcal{B}(y)$ satisfies the {\it modified Bessel equation}, 
\begin{equation} \label{e:modified Bessel equation}
y^2\cdot\frac{d^2 \mathcal{B}(y)}{dy^2}+y\cdot\frac{d\mathcal{B}(y)}{dy}-(y^2+\frac{1}{n^2})\cdot \mathcal{B}(y)=0.
\end{equation}
In the latter case $j_k\neq 0$, 
we make the following transformation
\begin{align}
\begin{cases}
y=-j_k \cdot \zeta \leq 0 \\
w_k(\zeta) = e^{\frac{j_k\cdot\zeta}{2}} \cdot \mathcal{J}(-j_k\cdot\zeta),
\end{cases}
\label{e:nonhomogeneous-transformation}
\end{align}
then $\mathcal{J}(y)$ satisfies the {\it confluent hypergeometric equation},
\begin{equation}
y\cdot \frac{d^2 \mathcal{J}(y)}{dy^2}+(\fa-y)\cdot \frac{d\mathcal{J}(y)}{dy}-\fb\cdot \mathcal{J}(y)=0,\label{e:confluent-hypergeometric-ode}
\end{equation}
where 
\begin{align}
\begin{cases}
\fa=1-\frac{1}{n}
\\
\fb= \frac{1}{2}(1-\frac{1}{n})-\frac{\lambda_k}{j_k\cdot n}.
\end{cases}
\end{align}
It is straightforward to see  that $\alpha\in(0,1)$ and $\beta\in(-\infty,0]$.

\begin{remark}
The above ODE transformations were first used by \cite{Koehler-Kuehnel}. 
\end{remark}

\begin{remark} The homogeneous equation \eqref{e:homogeneous-ODE} was studied by the authors in the special case  
 $n=\dim_{\dC}(\Ca)=2$. When $j_k=0$, \eqref{e:homogeneous-ODE} has standard solutions given by exponential functions. When $j_k>0$,  the transformation was chosen as 
 \begin{equation}
 \begin{cases}
 y=j_k^{\frac{1}{2}}\cdot z^{\frac{n}{2}}
 \\
 u_k(z)=e^{-\frac{j_kz^{n}}{2}}\cdot Q(j_k^{\frac{1}{2}}\cdot z^{\frac{n}{2}}).
 \label{e:H-transformation}\end{cases}
 \end{equation}
 We refer the readers to Section 4 of \cite{HSVZ} for more details.
In the special case $n=2$, $Q(y)$ is an Hermite function which satisfies the Hermite differential equation
\begin{equation}
\frac{d^2 Q(y)}{dy^2} - 2y\frac{dQ(y)}{dy} - 2(h+1) Q(y) = 0.
\end{equation}
The key tool to prove the estimates for $Q$ essentially relies on its integral representation formula. However, when $n>2$, if we perform the transformation as \eqref{e:H-transformation} then the resulting equation for $Q$ is more complicated to study. It turns out the transformation \eqref{e:nonhomogeneous-transformation} is a more suitable choice. 
\end{remark}

\section{The case of zero mode: uniform estimates and asymptotics}
\label{s:j=0}

In this subsection, we consider the case $j_k=0$ and corresponding eigenfunctions $\varphi_k$ are $S^1$-invariant on $Y^{2n-1}$. So \eqref{e:homogeneous-ODE} is reduced to the homogeneous ODE 
\begin{equation}
\frac{d^2u_k(z)}{dz^2}-n\lambda_k\cdot z^{n-2}u_k(z)=0, z\geq 1. 
\end{equation}
When $\lambda_k=0$ the equation has trivial solutions given by linear functions. In this subsection we always assume $\lambda_k\neq 0$. 
As discussed in Section \ref{s:ode-setup} under the change of variables given by \eqref{e:convert-to-linear} and \eqref{e:homogeneous-transformation}, we are led to study the modified Bessel equation.
\begin{equation}
y^2\cdot\frac{d^2 \mathcal{B}(y)}{dy^2}+y\cdot\frac{d\mathcal{B}(y)}{dy}-(y^2+\nu^2)\cdot \mathcal{B}(y)=0,\ \nu\in\dR.
\end{equation}
There are two linearly independent solutions $I_{\nu}(y)$ and $K_{\nu}(y)$
 called the {\it modified Bessel functions}, whose definition is given in Appendix \ref{s:appendix-1}. 
 These yield two linearly independent solutions to the original  equation \eqref{e:homogeneous-ODE}, given by 
\begin{equation}
\begin{cases}\mathcal{G}_k(z)\equiv z^{\frac{1}{2}} \cdot I_{\frac{1}{n}}\Big(2\sqrt{\frac{\lambda_k}{n}}\cdot z^{\frac{n}{2}}\Big),\\ \mathcal{D}_k(z)\equiv z^{\frac{1}{2}} \cdot K_{\frac{1}{n}}\Big(2\sqrt{\frac{\lambda_k}{n}}\cdot z^{\frac{n}{2}}\Big). 
\end{cases}
\label{e:fundamental-solution-bessel-type}
\end{equation}
First by the definition of $I_\nu$ and $K_\nu$ we can compute its Wronskian
\begin{proposition}\label{l:bessel-Wronskian} Let $\nu>0$ and $y>0$, then
\begin{equation}
\mathcal{W}(I_{\nu}(y), K_{\nu}(y))=-\frac{1}{y}.
\end{equation}

\end{proposition}

\begin{proof}

Since $I_{\nu}$ and $K_{\nu}$ satisfy 
\begin{align}
\frac{d}{dy}(y\cdot I_{\nu}'(y)) - (y+\frac{\nu^2}{y}) I_{\nu}(y)=0,
\\
\frac{d}{dy}(y\cdot K_{\nu}'(y)) - (y+\frac{\nu^2}{y}) K_{\nu}(y)=0.
\end{align}
This implies that
\begin{equation}
K_{\nu}(y)\cdot \frac{d}{dy}(y\cdot I_{\nu}'(y)) - I_{\nu}(y)\cdot \frac{d}{dy}(y\cdot K_{\nu}'(y)) = 0,
\end{equation}
and hence
\begin{equation}
\frac{d}{dy}\Big(y \cdot  \mathcal{W}(I_{\nu}(y),K_{\nu}(y))\Big) = 0.
\end{equation}
Therefore, $y\cdot \mathcal{W}(I_{\nu}(y),K_{\nu}(y))$ is a constant. 

Next, we will compute this constant which equals the limit of $y\cdot \mathcal{W}(I_{\nu}(y),K_{\nu}(y))$ as $y\to 0$.
By definition,
\begin{equation}
\lim\limits_{y\to 0}I_{\nu}(y)\Big/\Big(\frac{y^{\nu}}{\Gamma(\nu+1)\cdot 2^{\nu}}\Big) = 1,\ \lim\limits_{y\to 0}K_{\nu}(y)\Big/\Big(\frac{\pi}{2\sin(\nu\pi)}\cdot \frac{2^{\nu}\cdot y^{-\nu}}{\Gamma(1-\nu)}\Big) = 1.
\end{equation}
Notice that \begin{equation}\Gamma(\nu+1)\Gamma(1-\nu)=\nu\Gamma(\nu)\Gamma(1-\nu)=\frac{\nu\pi}{\sin(\nu\pi)},\end{equation} then it is straightforward that
\begin{equation}
\lim\limits_{y\to 0} y\cdot ( I_{\nu}(y)K_{\nu}'(y) - K_{\nu}(y)I_{\nu}'(y)) = -1.
\end{equation} 
This completes the proof. 
\end{proof}

\begin{corollary}\label{l:j=0-Wronskian} For any $z>0$,  we have
\begin{equation}
\mathcal{W}(\mathcal{G}_k(z),\mathcal{D}_k(z)) =  - \frac{n}{2}.
\end{equation}
\end{corollary}
 
\begin{proof}
 
Applying Lemma \ref{l:bessel-Wronskian} and the chain rule,  
 \begin{equation}
\mathcal{W}(\mathcal{G}_k(z),\mathcal{D}_k(z))= - z\cdot  n(\frac{\lambda_k}{n})^{\frac{1}{2}}\cdot z^{\frac{n}{2}-1} \cdot \frac{1}{2(\frac{\lambda_k}{n})^{\frac{1}{2}}z^\frac{n}{2}} = -\frac{n}{2}.
\end{equation}
\end{proof}

By Corollary \ref{c:bessel-asymp}, we also have the asymptotics of the solutions for each \emph{fixed} $k$. 
\begin{lemma}\label{l:j=0-grow-asymp}As $z\rightarrow \infty$ we have 
\begin{align}
\mathcal{G}_{k}(z) & \sim \frac{1}{2\sqrt{\pi}\cdot (\frac{\lambda_k}{n})^{\frac{1}{4}}}\cdot\frac{e^{2\sqrt{\frac{\lambda_k}{n}}\cdot z^{\frac{n}{2}}}}{ z^{\frac{n-2}{4}}},
\\
\mathcal{D}_k(z) &\sim\frac{\sqrt{\pi}}{2 (\frac{\lambda_k}{n})^{\frac{1}{4}}}\cdot\frac{e^{-2\sqrt{\frac{\lambda_k}{n}}\cdot z^{\frac{n}{2}}}}{ z^{\frac{n-2}{4}}}.
\end{align}
\end{lemma}

In our proof of Theorem \ref{t:liouville-theorem-calabi}, we need uniform estimates (with respect to $k$ and $z$) on $\mathcal G_k$ and $\mathcal D_k$. So in the following, we will prove  uniform estimates for $I_{\nu}(y)$ and $K_{\nu}(y)$ for all $y\geq 1$. 
Notice that, in this subsection we are interested in the case $j_k=0$ which corresponds to $\nu=\frac{1}{n}$. However, the following formulae and estimates work for general $\nu\in\dR$, and we shall need the case $\nu=-\frac{1}{n}$  in Section \ref{s:jk not zero}. 
We will apply appropriate integral representations of $I_{\nu}(y)$ and $K_{\nu}(y)$ to study their upper bounds and asymptotic behaviors. The following integral formulae will play a fundamental role in our estimates: Let $y>0$, then by Lemma \ref{l:infinite-integral},  we have

\begin{equation}
I_{\nu}(y)=\frac{1}{\pi}\int_{0}^{\pi}e^{y\cos\theta} \cos(\nu\theta) d\theta - \frac{\sin(\nu\pi)}{\pi}\int_0^{\infty}e^{-y\cosh t - \nu t} dt
\end{equation}
and
\begin{equation}
K_{\nu}(y) = \int_{0}^{\infty}e^{-y\cosh t}\cosh(\nu t)dt.
\end{equation}

\begin{proposition}\label{p:bessel-functions-estimate}
The following hold
\begin{enumerate}
\item For all $\nu\in \R$, there is a constant $C(\nu)>1$ such that 
\begin{align}
C^{-1}(\nu) \cdot \frac{e^{-y}}{\sqrt{y}} \leq  K_{\nu}(y)  \leq C(\nu) \cdot \frac{e^{-y}}{\sqrt{y}}, \qquad   y\geq 1\label{e:K nu bound};
\\
I_{\nu}(y) \leq 
\begin{cases}C(\nu) \cdot \frac{e^y}{\sqrt{y}}, & y\geq 1, \\
C(\nu)\cdot  y^\nu, & 0<y\leq 1.
\end{cases}\label{e:all-y-I}
\end{align}
\item For all $\nu>-1$, we have 
\begin{align}
I_\nu(y)\geq 
\begin{cases}
C(\nu)^{-1}\cdot \frac{e^y}{\sqrt{y}}, &  y\geq 1,\\
C(\nu)^{-1} \cdot y^\nu, & 0<y\leq 1. 
\end{cases}
\end{align}
\end{enumerate}
\end{proposition}

\begin{proof}

In the proof the constant $C(\nu)$ may vary from line to line. 
First we prove Item (1).
To start with, we prove the upper bound estimate for the solution
$K_{\nu}(y)$.
Notice that $\cosh(t)\geq 1+\frac{t^2}{2}$ for every $t\geq 0$, then 
\begin{eqnarray}
K_{\nu}(y)
& = &  \int_{0}^{\infty}e^{-y\cosh t}\cosh(\nu t) dt
\nonumber\\
&\leq& \int_{0}^{\infty}e^{-y(1+\frac{t^2}{2})}\cosh(\nu t) dt
\nonumber\\
&=&
\frac{e^{-y}}{2}\Big(\int_{0}^{\infty}e^{-\frac{yt^2}{2}+\nu t} dt + \int_{0}^{\infty}e^{-\frac{yt^2}{2}-\nu t} dt\Big).
\end{eqnarray}
Now we prove that, for $y\geq 1$ and $\nu\in\dR$, 
 \begin{equation}
\int_{0}^{\infty}e^{-\frac{yt^2}{2}+\nu t} dt 
\leq C(\nu)\cdot\frac{1}{\sqrt{y}}.\label{e:quadratic-exp-bound}
\end{equation}

It is by straightforward computation that  \begin{eqnarray}
\int_{0}^{\infty}e^{-\frac{yt^2}{2}+\nu t} dt 
&=& \int_{0}^{\infty}e^{-(\sqrt{\frac{y}{2}}t-\frac{\nu}{2}\sqrt{\frac{2}{y}})^2+\frac{\nu^2}{2y}}dt
\nonumber\\
&=& \sqrt{\frac{2}{y}}\cdot e^{\frac{\nu^2}{2y}}\int_{-\frac{\nu}{2}\sqrt{\frac{2}{y}}}^{\infty}e^{-\tau^2}d\tau,
\end{eqnarray}
where $\tau=\sqrt{\frac{y}{2}}t-\frac{\nu}{2}\sqrt{\frac{2}{y}}$. Notice that 
\begin{equation}
\int_{-\frac{\nu}{2}\sqrt{\frac{2}{y}}}^{\infty}e^{-\tau^2}d\tau\leq \int_{-\infty}^{\infty} e^{-\tau^2}d\tau = \sqrt{\pi}.
\end{equation}
 Moreover, the assumption $y\geq 1$ implies $e^{\frac{\nu^2}{2y}}\leq e^{\frac{\nu^2}{2}}$, so it holds  that
\begin{equation}
\int_{0}^{\infty}e^{-\frac{yt^2}{2}+\nu t} dt 
\leq C(\nu)\cdot\frac{1}{\sqrt{y}}.
\end{equation}
Similarly,  
\begin{equation}
\int_{0}^{\infty}e^{-\frac{yt^2}{2}-\nu t} dt \leq C(\nu)\cdot\frac{1}{\sqrt{y}}.\end{equation}
Therefore, we have 
\begin{equation}
 K_{\nu}(y) \leq C(\nu)\cdot\frac{e^{-y}}{\sqrt{y}},
\end{equation}
where $C(\nu)>0$ depends only on $\nu$.

Next we prove the lower bound estimate for $K_{\nu}(y)$. 
The integral representation of $K_{\nu}(y)$ can be written as follows,
\begin{align}
K_{\nu}(y) 
&=\frac{e^{-y}}{2}\Big(\int_{0}^{\infty}e^{-y(\cosh t-1)+\nu t} dt
+ \int_{0}^{\infty}e^{-y(\cosh t-1) - \nu t} dt \Big).
\end{align}
We will give lower bound estimates for the above two integrals respectively.
It is straightforward that
\begin{equation}
\int_{0}^{\infty}e^{-y(\cosh t-1)+\nu t} dt
\geq \int_0^1e^{-y(\cosh t-1)+\nu t} dt=\int_0^1e^{-\frac{y\cdot\cosh(\theta_t)t^2}{2}+\nu t}dt
\end{equation}
for some $0\leq \theta_t \leq 1$, which implies that
\begin{equation}
\int_{0}^{\infty}e^{-y(\cosh t-1)+\nu t} dt
\geq \int_0^1e^{-2yt^2 +\nu t}dt.
\end{equation}
The calculations in the last step imply that for $y\geq 1$,
\begin{equation}
\frac{C^{-1}(\nu)}{\sqrt{y}} \leq \int_0^{1} e^{-2yt^2+\nu t} \leq \frac{C(\nu)}{\sqrt{y}}.
\end{equation}
Therefore,
\begin{equation}
\int_{0}^{1}e^{-y(\cosh t-1)+\nu t} dt
\geq \frac{C^{-1}(\nu)}{\sqrt{y}}.
\end{equation}
By the same calculations, 
\begin{equation}
\int_{0}^{1}e^{-y(\cosh t-1)-\nu t} dt
\geq \frac{C^{-1}(\nu)}{\sqrt{y}}.
\end{equation}
This completes the proof of \eqref{e:K nu bound}.

To see \eqref{e:all-y-I} we first assume $y\geq 1$. We use the integral representation
\begin{equation}\label{eqn5-62}
I_{\nu}(y)=\frac{1}{\pi}\int_{0}^{\pi}e^{y\cos\theta} \cos(\nu\theta) d\theta - \frac{\sin(\nu\pi)}{\pi}\int_0^{\infty}e^{-y\cosh t - \nu t} dt.
\end{equation}
To estimate the second term, we use the integral estimate
\begin{equation}
\int_0^{\infty}e^{-y\cosh t -\nu t} dt \leq e^{-y} \int_0^{\infty}e^{-\frac{yt^2}{2}-\nu t}dt \leq C(\nu)\cdot\frac{e^{-y}}{\sqrt{y}}.
\end{equation}
Next, we estimate the first term of $I_{\nu}(y)$. Since  for every $\theta \in [0,\frac{\pi}{3}]$, 
\begin{equation}\cos\theta \leq 1 -\frac{\theta^2}{2} + \frac{\theta^4}{24}\leq 1-\frac{\theta^2}{4},\end{equation} 
then 
\begin{eqnarray}
\Big|\frac{1}{\pi}\int_{0}^{\pi}e^{y\cos\theta} \cos(\nu\theta) d\theta \Big| &\leq &\frac{1}{\pi}\int_{0}^{\frac{\pi}{3}}e^{y\cos\theta}  d\theta +  
 \frac{1}{\pi}\int_{\frac{\pi}{3}}^{\pi}e^{y\cos\theta}  d\theta 
\end{eqnarray}
Estimating the right hand side separately, we get 
\begin{eqnarray*}
\Big|\frac{1}{\pi}\int_{0}^{\pi}e^{y\cos\theta} \cos(\nu\theta) d\theta \Big| & \leq & \frac{e^y}{\pi}\int_{0}^{\frac{\pi}{3}}e^{-\frac{y\cdot\theta^2}{4}}d\theta  + \frac{2e^{\frac{y}{2}}}{3}
\leq  \frac{2e^y}{\sqrt{\pi}\cdot \sqrt{y}} + \frac{2e^{\frac{y}{2}}}{3}
\leq  \frac{10e^y}{\sqrt{y}}.
\end{eqnarray*}
Therefore,
\begin{equation}
I_{\nu}(y) \leq\frac{10e^y}{\sqrt{y}} +\frac{C(\nu)\cdot e^{-y}}{\sqrt{y }}\leq \frac{C(\nu)\cdot e^y}{\sqrt{y}}.
\end{equation}
Now we assume $y\in (0, 1]$. Since $I_\nu$ is smooth, 
we only need to analyze the behavior of $I_{\nu}(y)$ as $y\to0$. By the definition of $I_\nu(y)$ we see if $\nu\geq 0$ or $\nu$ is a negative integer, $\lim\limits_{y\rightarrow 0}I_\nu(y)=0$. For any $\nu<0$, we have
\begin{equation}
\lim\limits_{y\to0}I_{\nu}(y)\Big/\frac{(\frac{y}{2})^{\nu}}{\Gamma(\nu+1)}=1.
\end{equation}
Therefore, for any $y\in(0,1]$,
\begin{equation}
I_\nu(y)\leq C(\nu)\cdot y^{\nu}. 
\end{equation}

Now we prove Item (2). First we observe that by the definition of $I_\nu$ using power series, when $\nu\in (-1, 0)$, $I_\nu(y)$ is positive for all $y\in (0, \infty)$. So the lower bound of $I_\nu$ for $y\in (0, 1]$ follows just as before. Now we assume $y\geq 1$. 
To get the lower bound on $I_\nu$, it suffices to get the lower bound on the first term of \eqref{eqn5-62}. Suppose $\nu\neq 0$, denote $\eta_\nu=\min(\pi, \frac{\pi}{3|\nu|})$, then we divide the integral into two parts
\begin{equation}
\int_0^\pi e^{y\cos\theta}\cos(\nu\theta)d\theta=\int_0^{\eta_\nu} e^{y\cos\theta}\cos(\nu\theta)d\theta+\int_{\eta_\nu}^\pi e^{y\cos\theta} \cos(\nu\theta)d\theta.
\end{equation}
Since $\cos\theta\geq 1-\frac{\theta^2}{2}$  we get 
\begin{equation}
\int_0^{\eta_\nu} e^{y\cos\theta}\cos(\nu\theta)d\theta\geq \frac{1}{2}e^y \int_0^{\eta_\nu} e^{-\frac{\theta^2}{2}y}d\theta\geq C(\nu) \frac{e^y}{\sqrt{y}}, 
\end{equation}
 and for the second term we have
 \begin{equation}
 \Big|\int_{\eta_\nu}^\pi e^{y\cos\theta} \cos(\nu\theta)d\theta\Big|\leq \int_{\eta_\nu}^\pi e^{y\cos\theta} d\theta\leq (\pi-\eta_\nu)e^{\cos(\eta_\nu)y}. 
 \end{equation}
So we get 
\begin{equation}
I_\nu(y)\geq C^{-1}(\nu)\frac{e^y}{\sqrt{y}}. 
\end{equation}
For $\nu=0$ the argument is similar. 
 This completes the proof of Item (1).
\end{proof}

Converting the above back to $\mathcal G_k$ and $\mathcal D_k$, we obtain 

\begin{corollary} \label{p:j=0 uniform estimate for G and D} There is a dimensional constant $C(n)>1$ such that $\mathcal{D}_k$ and $\mathcal{G}_k$ yield the following estimates for all
 $z\in[ 2^{-\frac{2}{n}}\cdot n^{\frac{1}{n}}
 \cdot(\lambda_D)^{-\frac{1}{n}}, +\infty)$: 
\begin{align}
\frac{C^{-1}(n)}{\lambda_k^{\frac{1}{4}}}\cdot \frac{e^{-2\sqrt{\frac{\lambda_k}{n}}\cdot z^{\frac{n}{2}}}}{ z^{\frac{n-2}{4}}} &\leq \mathcal D_k(z)\leq \frac{C(n)}{\lambda_k^{\frac{1}{4}}}\cdot\frac{e^{-2\sqrt{\frac{\lambda_k}{n}}\cdot z^{\frac{n}{2}}}}{ z^{\frac{n-2}{4}}}, \label{e:Dk lower and upper bound}\\
\frac{C^{-1}(n)}{\lambda_k^{\frac{1}{4}}}\cdot \frac{e^{2\sqrt{\frac{\lambda_k}{n}}\cdot z^{\frac{n}{2}}}}{ z^{\frac{n-2}{4}}} & \leq \mathcal G_k(z)\leq  \frac{C(n)}{\lambda_k^{\frac{1}{4}}}\cdot\frac{e^{2\sqrt{\frac{\lambda_k}{n}}\cdot z^{\frac{n}{2}}}}{ z^{\frac{n-2}{4}}}. 
\end{align}
\end{corollary}

\section{The case of nonzero mode: uniform estimates and asymptotics}
 \label{s:jk not zero}
In this subsection, we consider the case $j_k\neq 0$ of the homogeneous equation
 \begin{equation}\frac{d^2 u_k(z)}{dz^2}-(\frac{j_k^2n^2}{4}\cdot z^{n}+n\lambda_k)z^{n-2}u_k(z)=0.\ z\geq 1,\label{e:homogeneous-generic}\end{equation}
In this case, corresponding eigenfunctions 
 $\varphi_k$ are not $S^1$-invariant on the fiber $Y^{2n-1}$.

 Under the change of variables given by \eqref{e:convert-to-linear} and \eqref{e:nonhomogeneous-transformation}, the above equation is transformed into
 the confluent hypergeometric equation,
 \begin{equation}
y\cdot \frac{d^2 \mathcal{J}(y)}{dy^2}+(\fa-y)\cdot \frac{d\mathcal{J}(y)}{dy}-\fb\cdot \mathcal{J}(y)=0,\ y< 0,\label{e:confluent-hypergeometric-equation}\end{equation}
where 
\begin{align}
\begin{cases}
\fa=1-\frac{1}{n}
\\
\fb= \frac{1}{2}(1-\frac{1}{n})-\frac{\lambda_k}{j_k\cdot n}.
\end{cases}
\end{align}
Since we have shown in Section \ref{s:ode-setup} that $\lambda_k \geq \frac{j_k(n-1)}{2}$, 
we have that 
\begin{equation}\beta \leq 0\ \text{and}\ \ \alpha-\beta\geq 1-\frac{1}{n}>0.\label{e:beta-non-positive}\end{equation} 
According to the discussion in Appendix \ref{s:appendix-1}, in our case $y<0$, the confluent hypergeometric
equation \eqref{e:confluent-hypergeometric-equation} has two linearly independent solutions
\begin{equation}
\Ku(\beta,\alpha,y)\equiv\sum\limits_{k=0}^{\infty}\frac{(\beta)_k}{(\alpha)_k}\cdot\frac{y^k}{k!}
\end{equation}
and 
\begin{equation}
\Tri(\beta,\alpha,y)\equiv \frac{e^y}{\Gamma(\fa-\fb)}\int_0^{\infty}e^{yt}t^{\fa-\fb-1}(1+t)^{\fb-1}dt.\label{e:def-Tri}
\end{equation}
By Item (3) of Lemma \ref{l:asymp-Ku-Tri}, as $y\to-\infty$, $\Tri(y)$ is a decaying solution to \eqref{e:confluent-hypergeometric-equation} for every $\alpha>\beta$, while Lemma \ref{l:general-Ku-asymp} shows that, in the case $\beta<0$, the solution $\Ku(y)$ is growing of certain polynomial rate as $y\to-\infty$.
These then yield two linearly independent solutions 
 to the homogeneous equation \eqref{e:homogeneous-generic}, \begin{equation}
 \begin{cases}
\mathcal{G}_k(z) = e^{\frac{j_k z^n}{2}} \cdot \Ku(\beta,\alpha,-j_kz^n),\\\mathcal{D}_k(z) = e^{\frac{j_k z^n}{2}} \cdot \Tri(\beta,\alpha,-j_kz^n).\label{e:generic-fundamental-solutions}
\end{cases}
\end{equation}
First we can compute the Wronskian
\begin{proposition}\label{p:generic-Wronskian}
For every $k\in\dN$, the Wronskian of $\mathcal{G}_k(z)$ and $\mathcal{D}_k(z)$ is a constant given by
\begin{equation}
\mathcal{W}(\mathcal{G}_k(z),\mathcal{D}_k(z)) = \frac{\Gamma(\alpha-1)}{\Gamma(\alpha-\beta)}\cdot j_k^{\frac{1}{n}}.
\end{equation}
\end{proposition}

\begin{proof}
Since $\mathcal{G}_k(z)$ and $\mathcal{D}_k(z)$ solve the homogeneous equation \begin{equation}\frac{d^2 u_k(z)}{dz^2}-(\frac{j_k^2n^2}{4}\cdot z^{n}+n\lambda_k)z^{n-2}u_k(z)=0
\end{equation}
which misses the first order term.
Immediately,  for all $z\geq 0$,
\begin{equation}\frac{d}{dz}\mathcal{W}(\mathcal{G}_k(z),\mathcal{D}_k(z))=0,
\end{equation}
which implies that the Wronskian $\mathcal{W}(\mathcal{G}_k(z),\mathcal{D}_k(z))$ is a constant. So it suffices to calculate it at $z=0$.
By the definition of the Wronskian, 
\begin{align}
\mathcal{W}(\mathcal{G}_k(z),\mathcal{D}_k(z))
=& e^{j_kz^n}
\cdot \Big(\Ku(\beta, \alpha, -j_kz^n)\cdot \frac{d}{dz}\Tri(\beta,\alpha,-j_kz^n)
\nonumber\\
&-  \frac{d}{dz}\Ku(\beta,\alpha, -j_kz^n)\cdot\Tri(\beta,\alpha,-j_kz^n)\Big).\label{e:Wronskian-expression}
\end{align}
To calculate $\frac{d}{dz}\Tri(\beta,\alpha,-j_kz^n)
$, we will apply Kummer's transformation law to relate
$\Tri$ and $\Ku$, that is, 
\begin{eqnarray}
&&\Tri(\beta,\alpha,-j_kz^n)
\nonumber\\
&=&e^{-{j_k}z^n}\cdot\mathcal{U}(\alpha-\beta,\alpha,{j_k}z^n)
\nonumber\\
&=&e^{-{j_k}z^n}\cdot \Big(\frac{\Gamma(1-\alpha)}{\Gamma(1-\beta)}\cdot\Ku(\alpha-\beta,\alpha,{j_k}z^n)
+\frac{\Gamma(\alpha-1)}{\Gamma(\alpha-\beta)}\cdot ({j_k}z^n)^{1-\alpha}\Ku(1-\beta,2-\alpha,{j_k}z^n)\Big)
\nonumber\\
&=&e^{-{j_k}z^n}\cdot \Big(\frac{\Gamma(1-\alpha)}{\Gamma(1-\beta)}\cdot\Ku(\alpha-\beta,\alpha,{j_k}z^n)
+\frac{\Gamma(\alpha-1)}{\Gamma(\alpha-\beta)}\cdot {j_k}^{\frac{1}{n}} z \cdot \Ku(1-\beta,2-\alpha,{j_k}z^n)\Big)
\nonumber\\
&=&\frac{\Gamma(1-\alpha)}{\Gamma(1-\beta)}\cdot\Ku(\beta,\alpha,-{j_k}z^n)+\frac{\Gamma(\alpha-1)}{\Gamma(\alpha-\beta)}\cdot {j_k}^{\frac{1}{n}}z\cdot \Ku(1-\alpha+\beta,2-\alpha,-{j_k}z^n).
\end{eqnarray}
So it follows that
\begin{eqnarray}
&&\frac{d}{dz}\Tri(\beta,\alpha,-{j_k}z^n)\nonumber\\
&=&\frac{\Gamma(1-\alpha)}{\Gamma(1-\beta)}\cdot \frac{d}{dz}\Ku(\beta,\alpha,-{j_k}z^n) 
+ \frac{\Gamma(\alpha-1)}{\Gamma(\alpha-\beta)}\cdot {j_k}^{\frac{1}{n}}\cdot \Big( \Ku(1-\alpha+\beta,2-\alpha,-{j_k}z^n)
\nonumber\\ &+& z \cdot\frac{d}{dz}\Ku(1-\alpha+\beta,2-\alpha,-{j_k}z^n)\Big).
\end{eqnarray}
Since $n\geq 2$, it directly follows from the definition of $\Ku$ that 
\begin{align}
\frac{d}{dz}\Big|_{z=0}\Ku(\beta,\alpha,-{j_k}z^n)  = 0 ,\\ \frac{d}{dz}\Big|_{z=0}\Ku(1-\alpha+\beta,2-\alpha,-{j_k}z^n) = 0.
\end{align}
Therefore, 
\begin{eqnarray}
\frac{d}{dz}\Big|_{z=0}\Tri(\beta,\alpha,-{j_k}z^n)&=&\frac{\Gamma(\alpha-1)}{\Gamma(\alpha-\beta)}\cdot {j_k}^{\frac{1}{n}}\cdot \Ku(1-\alpha+\beta,2-\alpha,0) \label{e:diff-Ku}\nonumber\\
&=&\frac{\Gamma(\alpha-1)\cdot {j_k}^{\frac{1}{n}}}{\Gamma(\alpha-\beta)}. 
\end{eqnarray}
Now evaluate \eqref{e:Wronskian-expression} at $z=0$, we have 
\begin{equation}
\mathcal{W}(\mathcal{G}_k,\mathcal{D}_k)(z) =\mathcal{W}(\mathcal{G}_k,\mathcal{D}_k)(0)=\frac{d}{dz}\Big|_{z=0}\Tri(\beta,\alpha,-{j_k}z^n)=\frac{\Gamma(\alpha-1)\cdot {j_k}^{\frac{1}{n}}}{\Gamma(\alpha-\beta)}.
\end{equation}
\end{proof}

Applying Lemma \ref{l:asymp-Ku-Tri} and Lemma \ref{l:general-Ku-asymp}, immediately we have the following asymptotics for the solutions $\mathcal{G}_k(z)$ and $\mathcal{D}_k(z)$ for \emph{fixed} $k$.

\begin{lemma}\label{l:generic-asymp}
For each fixed $k$, as $z\to+\infty$, we have
\begin{align}
\mathcal{G}_k(z) &\sim \frac{\Gamma(\alpha)}{\Gamma(\alpha-\beta)}\cdot(j_kz^n)^{-\beta}\cdot e^{\frac{j_kz^n}{2}}, \label{l:ku generic}
\\
\mathcal{D}_k(z) & \sim(j_kz^n)^{\beta-\alpha}\cdot e^{-\frac{j_kz^n}{2}}.\label{l:tri generic}
\end{align}
\end{lemma}

Again we need to derive  uniform estimates and asymptotic behavior for $\Ku$ and $\Tri$. The idea is to first estimate them in terms of certain integrals and then apply \emph{Laplace's method}. 
To start with, we need some preliminary calculations for $\Ku$ and $\Tri$. 

By definition, 
\begin{eqnarray}
\Tri(\fb,\fa,y) &=& \frac{e^y}{\Gamma(\alpha-\beta)}\int_0^{\infty} e^{yt+(\alpha-\beta-1)\log t + (\beta-1)\log(t+1)}dt
\nonumber\\
&=& \frac{e^y}{\Gamma(\alpha-\beta)}\int_0^{\infty} e^{yt+(\alpha-\beta-1)\log\frac{t}{t+1}}\cdot \frac{1}{(1+t)^{1+\frac{1}{n}}}dt.
\end{eqnarray}
For simplicity, 
we denote 
\begin{equation}
F(t) \equiv yt+(\fa-\fb-1)\log\frac{t}{t+1},\label{e:F-def}
\end{equation}
then
\begin{equation}
\Tri(\fb,\fa,y) = \frac{e^y}{\Gamma(\alpha-\beta)}\int_0^{\infty} e^{F(t)}\cdot \frac{1}{(1+t)^{1+\frac{1}{n}}}dt.\label{e:Tri-trans}
\end{equation}

Now we give both upper and lower bounds for $\Ku(\fb,\fa,y)$ by simpler exponential integrals. 
\begin{lemma}\label{l:Ku-both-sides-bound-integral} Let $y\leq -1$, then following holds,
\begin{equation}
C_n^{-1}\cdot \frac{ e^y(-y)^{\frac{1-2\alpha}{4}}}{\Gamma(\alpha-\beta)}\cdot\int_{\frac{1}{\sqrt{-y}}}^{\infty}e^{G(u)}du \leq \Ku(\fb,\fa,y) \leq C_n\cdot \frac{ e^y(-y)^{\frac{1-2\alpha}{4}}}{\Gamma(\alpha-\beta)}\cdot\int_0^{\infty}e^{G(u)}du,\label{e:Ku-simpler-int-bound}
\end{equation}
where \begin{equation}G(u)\equiv-u^2+2\sqrt{-y}u +(\alpha-2\beta-\frac{1}{2})\log u.\label{e:G-def}\end{equation}
\end{lemma}
\begin{proof}
To prove this estimate, we need the following integral representation formula for $\Ku(\fb,\fa,y)$,
\begin{equation}
\Ku(\fb,\fa,y)=\frac{\Gamma(\fa)}{\Gamma(\fa-\fb)}\cdot e^{y}(-y)^{\frac{1-\fa}{2}}\cdot \int_0^{\infty}e^{-t}\cdot t^{\frac{\fa-1}{2}-\fb}\cdot I_{\fa-1}(2\sqrt{-yt})dt.\label{e:int-Ku}
\end{equation}
The proof is included in Lemma \ref{l:general-Ku-int} of Appendix \ref{s:appendix-1}.

The key point in the proof of \eqref{e:Ku-simpler-int-bound} is to apply the estimate of $I_{\alpha-1}$ in Proposition \ref{p:bessel-functions-estimate}. 
 By definition, $\alpha= 1-\frac{1}{n}$ and hence $\alpha-1 = -\frac{1}{n} \geq -\frac{1}{2}$.
Applying the upper bound estimate of $I_{\alpha-1}$ in \eqref{e:all-y-I} of Proposition  \ref{p:bessel-functions-estimate},
\begin{eqnarray}
I_{\alpha-1}(2\sqrt{-yt})  = I_{-\frac{1}{n}}(2\sqrt{-yt})
&\leq& C_n\cdot \max\Big\{(2\sqrt{-yt})^{-\frac{1}{n}}, (2\sqrt{-yt})^{-\frac{1}{2}}\cdot e^{2\sqrt{-yt}}\Big\}
\nonumber\\
&\leq & C_n\cdot (-yt)^{-\frac{1}{4}}\cdot e^{2\sqrt{-yt}}.
\end{eqnarray}
Substituting the above in \eqref{e:int-Ku}, 
\begin{eqnarray}
\int_0^{\infty}e^{-t}\cdot t^{\frac{\fa-1}{2}-\fb}\cdot I_{\fa-1}(2\sqrt{-yt})dt
&\leq &  C_n\cdot \int_0^{\infty}e^{-t+2\sqrt{-yt}}\cdot t^{\frac{2\alpha-3}{4}-\beta}dt
\nonumber\\
&=& C_n\cdot \int_0^{\infty}e^{-t+2\sqrt{-yt} + (\frac{2\alpha-3}{4}-\beta)\log t}dt
\nonumber\\
&=&C_n\cdot \int_0^{\infty}e^{-u^2+2\sqrt{-y}u +(\alpha-2\beta-\frac{1}{2})\log u}du.\end{eqnarray}
Therefore, 
\begin{eqnarray}
 \Ku(\fb,\fa,y)
&\leq&C_n\cdot \frac{ e^y(-y)^{\frac{1-2\alpha}{4}}}{\Gamma(\alpha-\beta)}\cdot\int_0^{\infty}e^{-u^2+2\sqrt{-y}u +(\alpha-2\beta-\frac{1}{2})\log u}du.\label{e:Ku-int-upper}
\end{eqnarray}
Next, $\Ku$ can be also bounded below in a similar way. 
In fact, we consider the integral domain $t\geq \frac{1}{-y}$ with $y\leq -1$, then
\begin{equation}
I_{\alpha-1}(2\sqrt{-yt}) \geq C_n^{-1} \cdot\frac{e^{2\sqrt{-yt}}}{(-yt)^{\frac{1}{4}}},
\end{equation}
and hence
\begin{align}
\int_0^{\infty}e^{-t}\cdot t^{\frac{\fa-1}{2}-\fb}\cdot I_{\fa-1}(2\sqrt{-yt})dt
&\geq  \int_{\frac{1}{-y}}^{\infty}e^{-t}\cdot t^{\frac{\fa-1}{2}-\fb}\cdot I_{\fa-1}(2\sqrt{-yt})dt
\nonumber\\
&\geq C_n^{-1}\cdot \int_{\frac{1}{-y}}^{\infty}
e^{-t+2\sqrt{-yt} + (\frac{2\alpha-3}{4}-\beta)\log t}dt
\nonumber\\
&= C_n^{-1} \cdot\int_{\frac{1}{\sqrt{-y}}}^{\infty}e^{-u^2+2\sqrt{-y}u +(\alpha-2\beta-\frac{1}{2})\log u}du.
\end{align}
Therefore, 
\begin{equation}
\Ku(\fb,\fa,y) \geq C_n^{-1}\cdot \frac{ e^y(-y)^{\frac{1-2\alpha}{4}}}{\Gamma(\alpha-\beta)}\cdot\int_{\frac{1}{\sqrt{-y}}}^{\infty}e^{-u^2+2\sqrt{-y}u +(\alpha-2\beta-\frac{1}{2})\log u}du.
\end{equation}
\end{proof}
Now we set up a few notations for convenience. Let
\begin{equation}Q \equiv \alpha-\beta-1\geq -\frac{1}{n},\ \gamma_n\equiv \frac{1}{2}+\frac{1}{n},\end{equation} 
and recall the notations \eqref{e:F-def} and \eqref{e:G-def},
\begin{align}
F(t)& =   yt + Q \log\frac{t}{t+1}, 
\\
G(u)& = -u^2+2(-y)^{\frac{1}{2}}\cdot u+(2Q+\gamma_n)\cdot\log u.
\end{align}
By direct calculation
 \begin{align}
 F''(t) &= Q(-\frac{1}{t^2} + \frac{1}{(t+1)^2}),
 \\
 G''(u) &= -2 -\frac{2Q+\gamma_n}{u}.
 \end{align}
 Notice that $2Q+\gamma_n\geq \frac{1}{2}-\frac{1}{n}\geq 0$.
Therefore, $G(u)$ is strictly concave in $\dR_+$, and $F$ is strictly concave in $\dR$ if $Q>0$. 

We will split our analysis in two different cases:

{\bf Case (A):} $Q\geq 1$.

{\bf Case (B):} $Q\leq 1$.

Our main focus is Case (A) which is more difficult. The upper bound estimates in Case (B) follows from elementary integral calculations (see Lemma \ref{l:bounded-Q}).

\

{\bf Case (A)}

\

 Let $t_0>0$ be the unique critical point of $F(t)$ and 
let $u_0>0$ be the unique critical point of $G(u)$, then $t_0$ and $u_0$ satisfy the equations 
\begin{align}
t_0^2 + t_0 + \frac{Q}{y} = 0, \label{e:t_0-eq}
\\ 
u_0^2 - (-y)^{\frac{1}{2}}\cdot u_0 - \frac{2Q+\gamma_n}{2} =0.
\end{align}
Immediately we have
\begin{align}
t_0 &= \frac{-1 + \sqrt{1+\frac{4Q}{-y}}}{2},\label{e:t_0}
\\
u_0 &=\frac{(-y)^{\frac{1}{2}}}{2}\cdot  \Big(1+\sqrt{1+\frac{4Q}{-y}+\frac{2\gamma_n}{-y}}\Big).\label{e:u_0}
\end{align}

Now prove the following {\it effective estimates} on $\Ku$ and $\Tri$. The difference from Lemma \ref{l:generic-asymp} is here the estimates holds uniformly for all $\beta\leq 0$ (recall $\alpha$ is the fixed number $1-\frac{1}{n}$).

\begin{proposition}\label{l:laplace-method} 
 There exists some dimensional constant $C_n>0$ such that for every $y\leq -1$, the following estimates hold:
\begin{align}
  C_n^{-1}\cdot  Q^{-\frac{1}{4}-\frac{1}{2n}}\cdot\frac{(-y)^{-1}\cdot e^{y+F(t_0)}}{\Gamma(Q+1)}&\leq \Tri(\beta,\alpha,y) \leq
 C_n \cdot Q^{\frac{1}{4}} \cdot \frac{e^{y+F(t_0)}}{\Gamma(Q+1)},  \label{e:Tri-both-sides-uniform-estimate}
 \\
 C_n^{-1}\cdot Q^{-\frac{1}{4}}\cdot \frac{ (-y)^{\frac{1-2\alpha}{4}}\cdot e^{y+G(u_0)}}{\Gamma(Q+1)} &\leq  
  \Ku(\beta,\alpha,y) \leq C_n\cdot \frac{ (-y)^{\frac{1-2\alpha}{4}}\cdot e^{y+G(u_0)}}{\Gamma(Q+1)}. \label{e:Ku-both-sides-uniform-estimate}
\end{align}

\end{proposition}

\begin{proof}

Our main strategy is to apply  {\it  Laplace's method}. The basic idea is that the above exponential integrals are concentrated at the critical values $t_0$ and $u_0$. 

First, we prove the uniform estimate for 
$\Tri(\fb,\fa,y)$. By \eqref{e:Tri-trans},
\begin{equation}
\Tri(\fb,\fa, y) \leq \frac{e^y}{\Gamma(\fa-\fb)}\int_0^{\infty}e^{F(t)}dt.
\end{equation}

Clearly, the upper bound of $\Tri(\fb,\fa,y)$ follows from the upper bound estimate of $\int_0^{\infty}e^{F(t)}dt$. 
Write
\begin{equation}
\int_0^{\infty} e^{F(t)} dt  = \int_{0}^{2t_0}e^{F(t)}dt + \int_{2t_0}^{\infty} e^{F(t)}dt. \label{e:F-int}
\end{equation}
We will estimate the two terms separately. 

To estimate the first term in \eqref{e:F-int}, we make a change of variable \begin{equation}t=t_0 \cdot (1+\xi),\ \xi\in(-1,1),\end{equation}
then Taylor's theorem gives that
\begin{eqnarray}
F(t) - F(t_0) &=& F(t_0(1+\xi)) - F(t_0) 
\nonumber\\
&=& F'(t_0)\cdot t_0\cdot\xi + \frac{F''(\theta)} {2}\cdot t_0^2 \cdot \xi^2
\nonumber\\
&=& \frac{F''(\theta)} {2}\cdot t_0^2 \cdot \xi^2, \label{e:taylor}
\end{eqnarray}
where $\theta$ is between $t$ and $t_0$. Now we need to estimate the quadratic error term. 
It is straightforward calculation that
\begin{align}
F'''(t) &=  2(\frac{Q}{t^3}-\frac{Q}{(t+1)^3})>0,
\end{align}
then $F''(t)$ is increasing in $t$.
Since $\theta$ is between $t_0$ and $ t\in[0, 2t_0]$, the above monotonicity of $F''$ implies $F''(\theta)\leq F''(2t_0)<0$. 
So the first term of \eqref{e:F-int} becomes
\begin{eqnarray}
\int_0^{2t_0} e^{F(t)}dt 
&=&
 e^{F(t_0)} \int_0^{2t_0} e^{F(t)-F(t_0)} dt
 \nonumber\\
&\leq & e^{F(t_0)}\cdot t_0\cdot\int_{-1}^{1} e^{\frac{F''(2t_0)}{2}\cdot t_0^2 \cdot \xi^2} d\xi 
 \end{eqnarray}
By direct computations, 
$F''(2t_0)=-\frac{(4t_0+1)}{4t_0^2(2t_0+1)^2}\cdot Q$. So we have, 
 \begin{eqnarray}
 \int_0^{2t_0} e^{F(t)}dt 
 &\leq & e^{F(t_0)}\cdot t_0\cdot\int_{-1}^1 e^{-\frac{4t_0+1}{8(2t_0+1)^2}\cdot Q\cdot\xi^2} d\xi
\nonumber\\
 &\leq & C_n \cdot \frac{t_0(2t_0+1)}{\sqrt{4t_0+1}\cdot\sqrt{Q}}\cdot e^{F(t_0)}
 \nonumber\\
 &\leq & C_n \cdot Q^{\frac{1}{4}}\cdot e^{F(t_0)}, 
\end{eqnarray}
where we used that $t_0\leq C_n\cdot Q^{1/2}$ (since $y\leq -1$ and $Q\geq 1$).

Next, we estimate the second term in \eqref{e:F-int}.
Since we have proved $F''(t)<0$, so this implies that $F'(t)$ is decreasing and hence $F'(t)\leq F'(2t_0)$ for any $t\geq 2t_0$. Now Taylor's theorem gives that
\begin{equation}
F(t) \leq F(2t_0) + F'(2t_0) \cdot (t-2t_0),
\end{equation}
which implies that
\begin{equation}\int_{2t_0}^{\infty}e^{F(t)}dt 
\leq e^{F(2t_0)}\int_{2t_0}^{\infty}e^{F'(2t_0)\cdot(t-2t_0)}dt
= \frac{e^{F(2t_0)}}{-F'(2t_0)}.\end{equation}
One can check that $F'(2t_0)=\frac{y(3t_0+1)}{2(2t_0+1)}<0$ with $0<t_0<+\infty$. Since $F'(t)<0$ for all $t>t_0$, so $F(2t_0)\leq F(t_0)$ and hence for $y\leq -1$ we have
\begin{equation}
\int_{2t_0}^{\infty} e^{F(t)}dt \leq C_n e^{F(t_0)}.
\end{equation}
Combining the above, we have
\begin{equation}
\int_0^{\infty}e^{F(t)}dt\leq C_n\cdot Q^{\frac{1}{4}}\cdot e^{F(t_0)}.
\end{equation}
Therefore,
\begin{align}
\Tri(\beta,\alpha,y) &\leq
 C_n \cdot Q^{\frac{1}{4}} \cdot \frac{e^{y+F(t_0)}}{\Gamma(\alpha-\beta)}.
 \end{align}

 The lower bound estimate for $\Tri(\fb,\fa,y)$ also follows from Laplace's method and we just sketch the computations.  
\begin{align}
\Tri(\fb, \fa , y ) 
& = \frac{e^y}{\Gamma(\fa - \fb)}\int_0^{\infty}e^{F(t)}\cdot \frac{1}{(t+1)^{1+\frac{1}{n}}}dt
\nonumber\\
& \geq  \frac{e^y}{\Gamma(\alpha-\beta)}\int_{t_0(y)}^{2t_0(y)}e^{F(t)}\cdot \frac{1}{(t+1)^{1+\frac{1}{n}}}dt
\nonumber\\
&\geq \frac{ e^y}{\Gamma(\alpha-\beta)\cdot(1+2t_0)^{1+\frac{1}{n}}}\int_{t_0(y)}^{2t_0(y)}e^{F(t)}dt.
\end{align}
By the concavity of $F(t)$ and the monotonicity of $F''(t)$ in the domain $t_0\leq t\leq 2t_0$, we have
\begin{equation}
\int_{t_0(y)}^{2t_0(y)}e^{F(t)}dt\geq e^{F(t_0)}
\int_{t_0(y)}^{2t_0(y)}e^{\frac{F''(t_0)}{2}(t-t_0)^2}dt\geq C_n\cdot e^{F(t_0)} \frac{t_0(t_0+1)}{\sqrt{2t_0+1}\cdot \sqrt{Q}} 
\end{equation}
It is elementary to see that 
\begin{equation}
C_n Q^{\frac{1}{2}}(-y)^{-1}\leq t_0\leq C_n\cdot {Q^{\frac{1}{2}}}
\end{equation}
Therefore,
\begin{equation}
\Tri(\fb,\fa,y)\geq  C_n\cdot  Q^{-\frac{1}{4}-\frac{1}{2n}}\cdot\frac{e^y\cdot (-y)^{-1}}{\Gamma(\alpha-\beta)}\cdot e^{F(t_0)}.
\end{equation}

The uniform estimate for $\Ku(\fb,\fa, y)$ stated in \eqref{e:Ku-both-sides-uniform-estimate} can be proved in the same way. One just needs to apply Laplace's method to the integral estimate formula in Lemma \ref{l:Ku-both-sides-bound-integral}.
We can eventually obtain
\begin{equation}
 C_n^{-1} \cdot Q^{-\frac{1}{4}}\cdot  e^{G(u_0)} \leq 
\int_0^{\infty}e^{G(u)}du \leq C_n \cdot e^{G(u_0)}.
\end{equation}
We omit the computations here.
\end{proof}

Converting into the variables $z$, we obtain  

\begin{corollary}
\label{c:Dk upper and lower bound j not zero case}
There exists $C_n>0$ such that for all $z\geq 1$, we have 
\begin{align}
C_n^{-1} \cdot \frac{Q^{-\frac{1}{4}-\frac{1}{2n}}}{\Gamma(Q+1)} \cdot e^{-\frac{j_k\cdot z^{n}}{2}+F(t_0(z))} \cdot (j_k z^{n})^{-1} &\leq  \mathcal{D}_k(z) \leq
 C_n \cdot \frac{Q^{\frac{1}{4}}}{\Gamma(Q+1)} \cdot e^{-\frac{j_k\cdot z^{n}}{2}+F(t_0(z))},
 \\
 C_n^{-1}\cdot Q^{-\frac{1}{4}}\cdot\frac{(j_k\cdot z^n)^{\frac{1-2\alpha}{4}}}{\Gamma(Q+1)}\cdot e^{-\frac{j_k\cdot z^n}{2}+G(u_0(z))} &\leq   \mathcal{G}_k(z)
\leq C_n\cdot \frac{(j_k\cdot z^n)^{\frac{1-2\alpha}{4}}}{\Gamma(Q+1)}\cdot e^{-\frac{j_k\cdot z^n}{2}+G(u_0(z))},\end{align}
where  $Q\equiv\alpha-\beta-1\geq 1$.

\end{corollary}

The next Proposition essentially gives an estimate of the product of $\Ku$ and $\Tri$.

\begin{proposition}\label{l:product-estimate}There exists some dimensional constant $C_n>0$ such that
for any $y\leq -1$, we have 
\begin{equation}
e^{F(t_0) + G(u_0)} \leq C_n   (-y)^{\frac{\gamma_n}{2}} e^{-y}   e^{-Q}  Q^{Q+\frac{\gamma_n}{2}}.
\end{equation}
In particular we have 
\begin{equation} \label{e:sharp estimate for product}
\Tri\cdot \Ku\leq C_n\cdot \frac{\Gamma(\alpha)}{\Gamma(\alpha-\beta)^2}(-y)^{\frac{1}{n}} Q^Qe^{-Q} e^y.  
\end{equation}

\end{proposition}

\begin{proof} 
The calculation in the proof is purely elementary. 
The order estimate involving the parameter $Q$ will be used at crucial places for our later estimates, so we include the detailed proof. 
Plugging the critical points formulae \eqref{e:t_0} and \eqref{e:u_0} into the expression of $F$ and $G$, 
\begin{equation}
F(t_0) + G(u_0) = yt_0 + (-y)^{\frac{1}{2}} u_0 - \frac{2Q+\gamma_n}{2} + Q\log\frac{t_0}{t_0+1} + (2Q+\gamma_n)\log u_0,
\end{equation}
where $Q\equiv \alpha-\beta-1$ and $\gamma_n\equiv \frac{1}{2}+\frac{1}{n}$ as before.

First, it is straightforward that \begin{equation}
y t_0 + (-y)^{\frac{1}{2}} u_0 \leq  (-y) + \frac{\gamma_n}{2}. \end{equation}
So this implies that
\begin{eqnarray}
e^{F(t_0)+G(u_0)}&\leq& C_n\cdot e^{-y} \cdot e^{-Q} 
\cdot\Big(\frac{t_0}{1+t_0}\Big)^{Q}\cdot u_0^{2Q+\gamma_n}
\nonumber\\
&= &C_n\cdot e^{-y} \cdot e^{-Q} 
\cdot\Big(\frac{t_0^2}{t_0(1+t_0)}\Big)^{Q}\cdot u_0^{2Q+\gamma_n}
\nonumber\\
&= & C_n\cdot e^{-y}\cdot u_0^{\gamma_n} \cdot e^{-Q} \cdot \frac{(u_0t_0)^{2Q}}{(\frac{Q}{-y})^Q},
\end{eqnarray}
where the last equality follows from \eqref{e:t_0-eq}. 

Now we claim \begin{equation}
u_0t_0\leq(-y)^{-\frac{1}{2}}\cdot(Q+\frac{\gamma_n}{2}).\label{e:ut-product-estimate}
\end{equation}
To prove this,  we denote $\tau\equiv\frac{2\gamma_n}{-y}>0$ and $\hq\equiv\frac{4Q}{-y}>0$. Then
using the critical point formulae of $u_0$ and $t_0$ given by \eqref{e:t_0} and \eqref{e:u_0}, we obtain
\begin{eqnarray}&& u_0t_0 \nonumber\\
&=& \frac{(-y)^{\frac{1}{2}}}{4}\cdot \Big(1+\sqrt{1+\hq}\Big)\cdot\Big(-1+\sqrt{1+\hq + \tau}\Big)
\nonumber\\
 &=& \frac{(-y)^{\frac{1}{2}}}{4}\cdot \Big(-1+\sqrt{1+\hq}\cdot \sqrt{1+\hq+\tau} 
+ \sqrt{1+\hq} - \sqrt{1+\hq +\tau}\Big)
\nonumber\\
&\leq &\frac{(-y)^{\frac{1}{2}}}{4}\cdot \Big(-1+\sqrt{1+\hq+\tau}\cdot \sqrt{1+\hq+\tau} 
+ \sqrt{1+\hq+\tau} - \sqrt{1+\hq +\tau}\Big)
\nonumber\\
&= &\frac{(-y)^{\frac{1}{2}}}{4}\cdot (\hq+\tau)
 \nonumber\\
 &=&(-y)^{-\frac{1}{2}}\cdot(Q+\frac{\gamma_n}{2}).
\end{eqnarray}
Then it follows that
\begin{equation}
\frac{(u_0t_0)^{2Q}}{(\frac{Q}{-y})^Q} \leq  \frac{(Q+\frac{\gamma_n}{2})^{2Q}}{Q^Q} =Q^Q\cdot (1+\frac{\gamma_n}{2Q})^{2Q} \leq e^{\gamma_n}\cdot Q^{Q}.
\end{equation}
Moreover, we notice that
\begin{equation}
u_0^{\gamma_n} \leq C_n \cdot Q^{\frac{\gamma_n}{2}}\cdot (-y)^{\frac{\gamma_n}{2}}.\end{equation}
Therefore, combining all the above, we have
\begin{eqnarray}
e^{F(t_0) + G(u_0)} &\leq& C_n\cdot e^{-y}\cdot u_0^{\gamma_n} \cdot e^{-Q} \cdot \frac{(u_0t_0)^{2Q}}{(\frac{Q}{-y})^Q} 
\nonumber\\
&\leq& C_n\cdot   (-y)^{\frac{\gamma_n}{2}}\cdot e^{-y} \cdot  e^{-Q} \cdot Q^{Q+\frac{\gamma_n}{2}}.  
\end{eqnarray}
\end{proof}

In the next subsections, we will also need the following monotonicity formula to study the integral estimates for the above fundamental solutions $\mathcal{G}_k$ and $\mathcal{D}_k$.
\begin{lemma}\label{l:monotonicity}
Let
\begin{align}
\hf(z) &\equiv -\frac{jz^n}{2} + F(t_0(z)),
\\
\hg(z) &\equiv -\frac{jz^n}{2} + G(u_0(z)),
\end{align}
then for all $\eta\geq 0$,   when  $z\geq \eta^{\frac{2}{n}}$, 
   $\hf(z)+\eta \cdot z^{\frac{n}{2}}$ is decreasing and $\hg(z)-\eta \cdot z^{\frac{n}{2}}$ is increasing. 

\end{lemma}

\begin{proof}

Let $y = - jz^n$, then 
it is straightforward that
\begin{align}
\frac{d \widehat{F}(y)}{dy} = \frac{1}{2} + t_0(y) + F'(t_0(y))\cdot \frac{dt_0(y)}{dy}
 = \frac{1}{2} + t_0(y) 
= \frac{1}{2}\sqrt{1+\frac{4Q}{-y}} \geq \frac{1}{2}.\end{align}
This implies that, as $z\geq \eta^{\frac{2}{n}}$,
\begin{align}
\frac{d(\widehat{F}(z)+\eta z^{\frac{n}{2}})}{dz} = \frac{d\widehat{F}(y)}{dy}\cdot  (-nj\cdot z^{n-1}) + \frac{n\cdot \eta}{2}\cdot z^{\frac{n}{2}-1}
\leq -\frac{n}{2} \cdot z^{\frac{n}{2}-1} (j\cdot z^{\frac{n}{2}} -  \eta)
 \leq 0.
\end{align}
By similar calculations, one can also obtain that $\widehat{G}(z) - \eta\cdot z^{\frac{n}{2}}$
is increasing as $z\geq \eta^{\frac{2}{n}}$.
\end{proof}

\

{\bf Case (B)}: Now we consider the case when $Q\leq 1$. As  mentioned in the above, this case is easier.

\begin{lemma}
\label{l:bounded-Q} Let $Q\leq 1$, then there is some dimensional constant $C_n>0$ such that
\begin{align}
C_n^{-1}\cdot e^y \cdot (-y)^{\beta-\alpha} & \leq \Tri(\fb, \fa, y)  \leq  e^y \cdot (-y)^{\beta-\alpha},  \label{e:Tri-estimate-bounded-Q}
\\
C_n^{-1}\cdot(-y)^{-\beta} & \leq \Ku (\fb, \fa, y)  \leq C_n\cdot(-y)^{-\beta}. \label{e:Ku-estimate-bounded-Q}
\end{align}
for all $y\leq -1$.
\end{lemma}

\begin{remark}
In the case $Q\leq 1$, the estimate is optimal in the sense that it coincides with the asymptotic behavior of $\Tri$ and $\Ku$ for fixed $\alpha$ and $\beta$, as given in Lemma \ref{l:asymp-Ku-Tri} and Lemma  \ref{l:general-Ku-asymp}.
\end{remark}

\begin{proof}
First, we prove \eqref{e:Tri-estimate-bounded-Q}. Both the upper bound and lower bound estimates can be proved in the similar way:
\begin{align}
\Tri(\beta,\alpha,y)& =  \frac{e^y}{\Gamma(\fa-\fb)}\int_0^{\infty}e^{yt}t^{\fa-\fb-1}(1+t)^{\fb-1}dt
\nonumber\\
&\leq     \frac{e^y}{\Gamma(\fa-\fb)}\cdot \int_0^{\infty }e^{yt}t^{\fa-\fb-1}dt
\nonumber\\
& =   \frac{e^y\cdot (-y)^{\beta-\alpha}}{\Gamma(\fa-\fb)}\cdot \int_0^{\infty}e^{-u}u^{\fa-\fb-1}du
\nonumber\\
&=  e^y\cdot (-y)^{\beta-\alpha}.\end{align}
Similarly,
\begin{align}
\Tri(\fb,\fa, y) 
& \geq \frac{e^y}{\Gamma(\fa-\fb)}\int_0^1e^{yt}t^{\fa-\fb-1}(1+t)^{\fb-1}dt
\nonumber\\
& \geq C_n\cdot e^y \int_0^1 e^{yt}t^{\fa-\fb- 1}dt 
\nonumber\\
& \geq C_n \cdot e^y \cdot (-y)^{\fb-\fa}.
\end{align}

Next, we prove the upper bound estimate for $\Ku$. Notice in the proof of Lemma \ref{l:laplace-method} we do not need the condition $Q\leq 1$ for the upper bound on $\Ku$. So we have
\begin{equation}
\Ku(\beta,\alpha,y) \leq C_n\cdot \frac{\Gamma(\alpha)}{\Gamma(\alpha-\beta)}\cdot (-y)^{\frac{1-2\alpha}{4}}\cdot e^{y+G(u_0)}.
\end{equation}
 To prove \eqref{e:Ku-estimate-bounded-Q}, we need an upper bound estimate for $e^{y+G(u_0)}$. This follows from elementary computations.
In fact,
\begin{equation}
e^{y+G(u_0)} = e^{y  - u_0^2 + 2\sqrt{-y} u_0}\cdot (u_0)^{2Q+\gamma_n} 
\nonumber\\
 \leq C_n \cdot  e^{y  - u_0^2 + 2\sqrt{-y} u_0}\cdot (-y)^{Q+\frac{\gamma_n}{2}}.
\end{equation}
Notice that $u_0$ satisfies $G'(u_0)=0$, i.e.,
\begin{equation}
u_0^2 - \sqrt{-y}\cdot u_0 -\frac{2Q+\gamma_n}{2} = 0,
\end{equation}
so we have
\begin{equation}
e^{y+G(u_0)} \leq C_n \cdot e^{y+\sqrt{-y}u_0}\cdot (-y)^{Q+\frac{\gamma_n}{2}}.
\end{equation}
By \eqref{e:u_0}, it is straightforward that
\begin{align}
y+\sqrt{-y}u_0
=\frac{y}{2}\Big(1-\sqrt{1+\frac{4Q+2\gamma_n}{-y}}\Big)
=\frac{2Q+\gamma_n}{1+\sqrt{1+\frac{4Q+2\gamma_n}{-y}}} 
\in [C_n^{-1}, C_n],
\end{align}
for some dimensional constant $C_n>0$.
Therefore,
\begin{align}
e^{y+G(u_0)} \leq C_n (-y)^{Q+\frac{1}{4}+\frac{1}{2n}},
\end{align}
and hence
\begin{align}
\Ku(\fb,\fa,y) \leq C_n (-y)^{Q+\frac{1}{n}} = C_n(-y)^{-\beta}.
\end{align}
This completes the proof. 
\end{proof}

Converting into the variables $z$ we obtain 

\begin{corollary}
\label{c:Dk upper and lower bound Q small case}
There exists $C_n>0$ such that for all $z\geq 1$, we have 
\begin{align}
C_n^{-1} \cdot e^{-\frac{j_k\cdot z^{n}}{2}} \cdot (j_k z^{n})^{\beta-\alpha} &\leq  \mathcal{D}_k(z)  \leq
 C_n \cdot e^{-\frac{j_k\cdot z^{n}}{2}} \cdot (j_k z^{n})^{\beta-\alpha},\\
C_n^{-1}\cdot e^{\frac{j_k\cdot z^n}{2}}\cdot (j_k z^n)^{-\beta} &\leq  \mathcal{G}_k(z)
 \leq C_n\cdot e^{\frac{j_k\cdot z^n}{2}}\cdot (j_k z^n)^{-\beta}.
\end{align}

\end{corollary}

We end this subsection by making some remarks regarding the above estimates on $\Ku$ and $\Tri$. Notice that in the case $Q\equiv \fa-\fb -1 \leq 1$  we applied Laplace's method to turn the problem into estimates on exponential integrals. One may wonder  how far the  uniform estimates in Lemma \ref{l:laplace-method} is from optimal comparing to the {\it non-uniform} estimate with the optimal order in Lemma \ref{l:bounded-Q}.  We can consider two extreme cases depending on the size of $Q$ compared with $-y$. 

First we assume $\frac{Q^2}{-y}\ll1$, which obviously includes the case when we fix $Q$ and let $y\rightarrow-\infty$. Then by definition we see that 
\begin{equation}
t_0=\frac{Q}{-y}+O\Big((\frac{Q}{-y})^2\Big), 
\end{equation}
and we get 
\begin{equation}
F(t_0)=yt_0+Q\log \frac{t_0}{t_0+1}=-Q+Q\log Q-Q\log(-y)+O(\frac{Q}{-y}).  
\end{equation}  
So by Lemma \ref{l:laplace-method} we get 
\begin{equation}
C_n^{-1} Q^{-\frac{1}{4}-\frac{1}{2n}} e^y (-y)^{-Q-1} Q^Qe^{-Q}\leq   \Tri\leq C_n\frac{1}{\Gamma(\alpha-\beta)}e^y (-y)^{-Q} Q^Qe^{-Q}Q^{\frac{1}{4}}.
\end{equation}
Notice by Stirling's formula for $Q$ large 
$\Gamma(\alpha-\beta)=Q\Gamma(Q)$ is comparable to 
$C_n Q^{\frac{3}{2}} Q^Qe^{-Q}$.  So up to polynomial errors in $Q$ this estimate is optimal comparing with \eqref{eqn-A27}. 
Similarly, we have 
\begin{equation}
u_0=(-y)^{\frac{1}{2}}\Big(1+\frac{Q+\frac{1}{2}\gamma_n}{-y}+O((\frac{Q}{-y})^2)\Big), 
\end{equation}
and 
\begin{equation}
G(u_0)=-u_0^2+2(-y)^{\frac{1}{2}}u_0+(2Q+\gamma_n)\log u_0\leq C_n e^{-y}(-y)^{Q+\frac{\gamma_n}{2}}.
\end{equation}
So 
\begin{equation}
\Ku\leq C_n \frac{\Gamma(\alpha)}{\Gamma(\alpha-\beta)}(-y)^{\frac{1-2\alpha}{4}} (-y)^{Q+\frac{\gamma_n}{2}}=C_n \frac{\Gamma(\alpha)}{\Gamma(\alpha-\beta)}(-y)^{-\beta},
\end{equation}
which is again optimal comparing with \eqref{eqn-A34}.

Secondly we assume the other extreme $\frac{Q}{(-y)^3}\gg1$. In this case we have 
\begin{equation}
t_0=\sqrt{\frac{Q}{-y}}-\frac{1}{2}+O(\sqrt{\frac{-y}{Q}}). 
\end{equation}
Then we get 
\begin{equation}
F(t_0)=-2\sqrt{-Qy}-\frac{1}{2}y+O(1), 
\end{equation}
and 
\begin{equation}
C_n^{-1}Q^{-\frac{1}{4}-\frac{1}{2n}} e^{\frac{1}{2}y-\sqrt{-Qy}} (-y)^{-Q-1}\leq  \Tri(y) \leq C_n\frac{1}{\Gamma(\alpha-\beta)}Q^{\frac{1}{4}}(-y)^{-Q}e^{\frac{1}{2}y-2\sqrt{-Qy}}. 
\end{equation}
Similarly, 
we get 
\begin{equation}
G(u_0)=2\sqrt{-Qy}-\frac{y}{2}+(Q+\frac{1}{2}\gamma_n)\log Q-Q.
\end{equation}
So 
\begin{equation}
\Ku(y) \leq C_n \cdot \frac{\Gamma(\alpha)}{\Gamma(\alpha-\beta)}(-y)^{\frac{1-2\alpha}{4}}e^{\frac{1}{2}y+2\sqrt{-Qy}} e^{-Q}Q^{Q+\frac{1}{2}\gamma_n}.
\end{equation}
In this case even though in the produce $\Tri\cdot \Ku$ there is a good cancellation each of them does behave quite differently from the previous case. This also gives a reason why we do get an optimal estimate (up to polynomial errors in $y$ and $Q$) for the product $\Tri\cdot \Ku$, comparing with \eqref{eqn-A27} and \eqref{eqn-A34}.

\section{Asymptotics of harmonic functions on the Calabi model space}
\label{s:asymp harmonic Calabi model}

 As Section \ref{s:ode-setup}, we fix $r_0\in (0, 1)$, and  view the Calabi model space $\mathcal C^n$ as the product of a fixed cross section $Y^{2n-1}\cong \{\varrho=r_0\}$ with the restricted metric $h_0=g_{\mathcal C^n}|_{\{\varrho=r_0\}}$ with  a ray $\mathbb R^+$.  The spectrum of the Laplacian operator on $Y$ is given by $\{\Lambda_k\}_{k=0}^\infty$, with $\Lambda_0=0$, and we have chosen an orthonormal basis of complex valued eigenfunctions of the form $\{\vf_k\}_{k=0}^{\infty}$ such that
 \begin{align}
 \begin{cases}
- \Delta_{Y^{2n-1}}\varphi_k=\Lambda_k\cdot \varphi_k,
 \\
 \|\varphi_k\|_{L^2(Y^{2n-1})}=1.
 \end{cases}
 \end{align}

In the asymptotic analysis of the harmonic functions on 
$\Ca$,  
we need some uniform estimates for the eigenfunctions $\{\varphi_k\}_{k=0}^{\infty}$ in the $L^2$-orthonormal basis.  
In particular, we need the following uniform $C^k$-estimate of the eigenforms in terms of the corresponding eigenvalues. The proof follows from the standard $W^{2,p}$-elliptic regularity and the Sobolev embedding theorems, so we omit it.
\begin{lemma}
\label{l:eigenfunction-bound}
Let $(M^m,g)$ be a closed Riemannian manifold of dimension $m\geq 2$. For any $p\in\dN$, denote by $\Lambda^{(p)}\equiv \{\lambda_j\}_{j=0}^{\infty}$ with $\lambda_0=0$ the spectrum of the Hodge Laplacian $\Delta$ acting on the $p$-forms.  For any $k\in\dN $, there is some constant $C>0$ depending only on $(M, g)$ and $k$,  $p$  such that for all  
$\phi_j\in\Omega^p(M^m)$ satisfying
\begin{align}
\begin{cases}
\Delta \phi_j = \lambda_j \phi_j,
\\
\|\phi_j\|_{L^2(M^m)}=1,
\end{cases}
\end{align}
we have
\begin{equation}
\| \nabla^k\phi_j \|_{C^k(M^m)}\leq C_k \cdot (\lambda_j)^{\frac{1}{2}[\frac{m}{2}]+\frac{k+1}{2}}.
\end{equation}

\end{lemma}

In addition, we need a basic lemma on the decay of Fourier coefficients of the expansion of a sufficiently smooth function in terms of eigenfunctions.

\begin{lemma}
\label{l:error-estimate}

 Let $K_0\geq 1$ and let      
 $\xi \in C^{2K_0}(Y^{2n-1})$ satisfy the $L^2$-expansion
 \begin{equation}
 \xi(\by) = \sum\limits_{k=1}^{\infty} \xi_k  \cdot \vf_k (\by),
 \end{equation}
then for all $k\in\dZ_+$,\begin{equation}|\xi_{k}|\leq  \frac{ C|\xi|_{C^{2K_0}(Y^{2n-1})}}{(\Lambda_{k})^{K_0}},\end{equation}
where the constant $C>0$ is independent of $k$.
\end{lemma}

\begin{proof}
 The estimate is proved by the standard integration by parts. Since the eigenfunctions $\vf_k$ satisfy 
 \begin{equation}
 -\Delta_{h_0}\vf_{k}=\Lambda_{k}\cdot\vf_{k}
 \end{equation} and $\|\vf_k\|_{L^2(Y^{2n-1})}=1$, we have that
\begin{align} 
|\xi_{k}(z)|&=\Big|\int_{Y^{2n-1}}\xi\cdot\vf_k\Big|
=\Big|\int_{Y^{2n-1}}\xi\cdot\frac{(-\Delta_{h_0})^{K_0}\vf_k}{(\Lambda_k)^{K_0}}\Big|\nonumber\\
&\leq \frac{1}{(\Lambda_k)^{K_0}}\int_{Y^{2n-1}}|\Delta_{h_0}^{K_0}\xi|\cdot|\vf_k|\\
&\leq \frac{C|\xi|_{C^{2K_0}(Y^{2n-1})}}{(\Lambda_k)^{K_0}},\end{align}
where $C>0$ depends only on the geometry of $Y^{2n-1}$. 
\end{proof}

\begin{proposition}
[Asymptotics of harmonic functions]\label{p:harmonic-function-decay} Let $(\Ca,g_{\Ca})$ be a Calabi model space with $\dim_{\dC}(\Ca)=n$. Define a constant
\begin{equation}
\label{e:definition of deltab}\delta_{b}\equiv 2\cdot \left(\frac{\lambda_D}{n}\right)^{\frac{1}{2}}>0
\end{equation}
where $\lambda_D>0$ is given by \eqref{e:definition of underline lambda}.  If $u$ is a harmonic function outside a compact set in $\mathcal C^n$ satisfying
\begin{equation}
|u(z,\by)| = O(e^{\delta\cdot z^{\frac{n}{2}}}) \label{e:slow-exp-growth}
\end{equation}
for some $\delta\in (0, \delta_b)$ as $z\rightarrow\infty$. Then $u$ can be decomposed as 
\begin{equation}
u(z,\by)= L (z) + h(z,\by)
\end{equation}
with the following properties:
\begin{enumerate}
\item $ L (z)= \kappa_0\cdot z + c_0$
for some $\kappa_0, c_0\in\dR$.
\item $h(z,\by)$ is harmonic and for any $k\in\dN$, there is some $C_k>0$ such that \begin{equation}|\nabla^k h(z,\by)|\leq C_k\cdot e^{-\underline\delta\cdot z^{\frac{n}{2}}}\label{e:harmonic-error-higher-estimate}\end{equation} for all $\ldel\in (0, \delta_b)$, as $z\to+\infty$. 

\end{enumerate}

\end{proposition}
 
 \begin{proof}

The proof consists of two steps.

In the first step, we will apply separation of variables to show that if a harmonic function $u$ satisfies \eqref{e:slow-exp-growth}, then $u(z,\by)=k_0\cdot z + c_0 + h(z,\by)$
for some $k_0,c_0\in\dR$ and $h(z,\by)$ has some exponential decaying rate.

Since $u$ is smooth, for any fixed $z\geq 1$, we have the fiber-wise $L^2$-expansion of $u$ as follows,
 \begin{equation}
 u(z,\by)=\sum\limits_{k=1}^{\infty}u_{k}(z)\cdot\vf_{k}(\by),\label{e:u-phi}
 \end{equation}
where $\by\in Y^{2n-1}$ and $u_{k}$ satisfies the equation 
\begin{equation}
\frac{d^2 u_k(z)}{dz^2}-(\frac{j_k^2n^2}{4}\cdot z^{n}+n\lambda_k)z^{n-2}u_k(z)=0,\ z\geq 1,
\end{equation}
for some $j_k\in\dN$
and $\lambda_k\geq0$. 
Notice that the expansion \eqref{e:u-phi}  converges in the $C^{\infty}$-topology. This follows from Lemma \ref{l:error-estimate}, Lemma \ref{l:eigenfunction-bound} and the Weyl law for spectrum asymptotics. 

For $k=0$ we have $j_k=\lambda_k=0$, and $u_k$ is a linear function of the form $\kappa_0\cdot z+c_0$. For $k\geq 1$, we can write $u_k$ as a linear combination of the two linearly independent solutions discussed in Section \ref{s:j=0} and \ref{s:jk not zero}. 
\begin{equation}
u_k(z)= C_k \cdot \mathcal{D}_k(z) + C_k^* \cdot  \mathcal{G}_k(z),
\end{equation}
where $\mathcal G_k$ is a growing  and $\mathcal D_k$ is decaying.  

We claim $C_k^*=0$ for all $k\in\dZ_+$. To see this, we apply Lemma \ref{l:error-estimate} to $u(z, \by)$, then for all $k\in\dZ_+$
\begin{equation}
|u_k(z)|=O(e^{\delta z^{\frac{n}{2}}}). 
\end{equation}
So the claim follows from the asymptotics of $\mathcal G_k(z)$ in Lemma \ref{l:j=0-grow-asymp} and \ref{l:generic-asymp} which corresponds to $j_k=0$ and $j_k\in\dZ_+$ respectively.

Now we define
\begin{equation}
h(z,\by) \equiv u(z,\by) - (\kappa_0\cdot z + c_0)=\sum_{k= 1}^{\infty} u_k(z)\cdot \vf_k(\by). 
\end{equation}
It suffices to show $h(z, \by)$ decays at the desired rate. Let $z_0>(2\delta_b)^{2/n}$  be sufficiently big so that $u$ is defined on $\{z\geq z_0\}$. 
Now we fix $K_0\equiv 2n+1$. 
Applying Lemma \ref{l:error-estimate} to $u(z_0, \by)$ we get  for all $k\in\dZ_+$,
\begin{equation}
|u_k(z_0)|\leq C_1 (\Lambda_k)^{-K_0}.
\end{equation}
We  separate in several cases. 
First, we consider $k\in\dZ_+$ with $j_k=0$.
Applying \eqref{e:Dk lower and upper bound}, then for any $\epsilon>0$ with $\ldel=(1-\epsilon)\delta_b<\delta_b $, if  $z\geq \frac{1}{\epsilon^{\frac{2}{n}}}\cdot z_0$,  
\begin{equation}
\Big|\frac{u_k(z)}{u_k(z_0)}\Big|=\Big|\frac{\mathcal{D}_k(z)}{\mathcal{D}_k(z_0)}\Big|\leq C e^{-2\sqrt{\frac{\lambda_k}{n}}\cdot (z^{\frac{n}{2}}-z_0^{\frac{n}{2}})}\leq C e^{-(1-\epsilon)\delta_b \cdot z^{\frac{n}{2}}}=C e^{-\ldel \cdot z^{\frac{n}{2}}}.\end{equation}
This implies that \begin{eqnarray}
\Big|\sum_{\substack{k>0\\ j_k=0}}\frac{u_k(z)}{u_k(z_0)}\cdot u_k(z_0)\cdot \vf_{k}(\by)\Big|
&\leq& 
\sum_{\substack{k>0\\ j_k=0}}\Big|\frac{u_k(z)}{u_k(z_0)}\Big| \cdot |u_k(z_0) |\cdot |\vf_{k}(\by)| 
\nonumber\\
&\leq& C
e^{-\delta_b\cdot z^{\frac{n}{2}}}\cdot\sum_{\substack{k>0\\ j_k=0}} \frac{1}{(\Lambda_{k})^{K_0-\frac{n}{2}}},
\end{eqnarray}
where the eigenfunction estimate
\begin{equation}
\|\varphi_k\|_{L^{\infty}(Y^{2n-1})} \leq C\cdot (\Lambda_k)^{\frac{n}{2}}.
\end{equation}
 follows from Lemma \ref{l:eigenfunction-bound}.

When $j_k\in\dZ_+$ we divide into two cases. When $Q\geq 1$ we  apply Corollary \ref{c:Dk upper and lower bound j not zero case} and Lemma \ref{l:monotonicity} (with $\eta=2\delta_b$) to get 

\begin{eqnarray}
\Big|\sum_{\substack{ j_k\geq 1\\ Q\geq 1}}\frac{u_k(z)}{u_k(z_0)}\cdot u_k(z_0)\cdot \vf_{k}(\by)\Big| 
& \leq & 
\sum_{\substack{ j_k\geq 1\\ Q\geq 1}}\Big|\frac{u_k(z)}{u_k(z_0)}\Big| \cdot |u_k(z_0)| \cdot |\vf_{k}(\by) | 
\nonumber\\
&\leq & Ce^{-\delta_b \cdot z^{\frac{n}{2}}}\sum_{\substack{ j_k\geq 1\\ Q\geq 1}} \frac{1}{(\Lambda_k)^{K_0-\frac{n}{2}-1}}.
\end{eqnarray}
Now when $Q\leq 1$ we apply instead Corollary \ref{c:Dk upper and lower bound Q small case} to get 
\begin{eqnarray}
\Big|\sum_{\substack{ j_k\geq 1\\ Q\leq 1}}\frac{u_k(z)}{u_k(z_0)}\cdot u_k(z_0)\cdot \vf_{k}(\by)\Big| 
& \leq & 
\sum_{\substack{j_k\geq 1\\ Q\leq 1}}\Big|\frac{u_k(z)}{u_k(z_0)}\Big| \cdot |u_k(z_0)| \cdot |\vf_{k}(\by) | 
\nonumber\\
&\leq & Ce^{- \frac{z^n}{2}}\sum_{\substack{ j_k\geq 1\\ Q\leq 1}} \frac{1}{(\Lambda_k)^{K_0-\frac{n}{2}-1}}.
\end{eqnarray}
Summing up all the above we get
\begin{equation}
|h(z, \by)|\leq Ce^{-\delta_b \cdot z^{\frac{n}{2}}} \sum_{k=1}^\infty \frac{1}{(\Lambda_k)^{K_0-\frac{n}{2}-1}}. 
\end{equation}
Since $K_0=2n+1$ we see the series converges. So the proof of the first step is done.

The second step is to prove the higher decaying estimate for the error function $h(z,\by)$, which follows from the uniform Schauder estimate. 
We have proved that the error function $h(z,\by)$ as a harmonic function satisfies
\begin{equation}
|h(z,\by)| \leq C_0 \cdot e^{-\ldel\cdot z^{\frac{n}{2}}}.
\end{equation}
By explicit and straightforward computations, a Calabi space $(\Ca, g_{\Ca})$ is collapsing  with bounded curvatures as $z\to+\infty$. We just lift the harmonic function $h$ to the local universal cover which is non-collapsed with uniformly bounded geometry.  So the following Schauder estimate holds for any $k\in \dZ_+$ and $\alpha\in(0,1)$ on the local universal cover,
\begin{equation}
|h|_{C^{k,\alpha}(B_{r_0}(\bx))} \leq C_k\cdot |h|_{C^0(B_{2r_0}(\bx))} \leq C_k\cdot e^{-\ldel\cdot z^{\frac{n}{2}}}.
\end{equation}
where $r_0>0$
 is some fixed constant of some definite size which is independent of $\bx\in\Ca$.
 In particular, at the center $\bx=(z,\by)$, we have
 \begin{equation}
|\nabla^k h(z,\by)| \leq C_k\cdot |h|_{C^0(B_{2r_0}(\bx))} \leq C_k\cdot e^{-\ldel\cdot z^{\frac{n}{2}}}.
\end{equation}
This completes the proof of \eqref{e:harmonic-error-higher-estimate}. 
\end{proof}

The above proposition
has the following  corollaries stating the Liouville type results on the incomplete Calabi space under both Neumann and Dirichlet boundary conditions. 
\begin{corollary}[Neumann boundary]\label{c:Liouville-neumann}
Let $(\Ca, g_{\Ca})$ be an incomplete Calabi model space based over a compact Calabi-Yau manifold $(D,\omega_D)$ and the natural moment map coordinate $z$ as in \eqref{e:moment-map-definition}, so that $\Ca$ is diffeomorphic to a topological product $[z_0,+\infty)\times Y^{2n-1}$ for some $z_0>0$ under the moment map coordinate $z$, where $Y^{2n-1}$ is a circle bundle over $D$. 
 Let $u$ be a solution of the Neumann boundary problem \begin{align}
\begin{cases}
\Delta_{g_{\Ca}} u(\bx) = 0, & \bx\in\Ca,
\\
\frac{\p u}{\p z}(\bx) = \kappa_0, & z(\bx)= z_0,
\end{cases}
\end{align}
where $\kappa_0\in\dR$.
If $u$ satisfies the growth condition $|u(\bx)| = O(e^{\delta\cdot z(\bx)^{\frac{n}{2}}})$ for some  $0<\delta<2(\frac{\lambda_D}{n})^{\frac{1}{2}}$, 
then  there is some $\ell_0\in\dR$ such that   i.e., $u=\kappa_0\cdot z + \ell_0$ on $\Ca$. In particular, $u$ is constant on $\Ca$ when $\kappa_0=0$.

\end{corollary}

\begin{remark}
This corollary is used in the bubbling analysis of the incomplete Calabi-Yau metrics in \cite{SZ} (See the proof of Proposition 5.12 in \cite{SZ}). \end{remark}

\begin{proof}

If $\delta>0$ satisfies $\delta\in(0,\delta_b)$ with $\delta_b$ defined in \eqref{e:definition of deltab}, then it directly follows from  Proposition \ref{p:harmonic-function-decay} and the proof that, the harmonic function $u$ has the expansion 
\begin{equation}
u(z,\by) = \kappa  \cdot z
+ \ell_0 + \sum\limits_{k=1}^{\infty} c_k \cdot \mathcal{D}_k(z)\cdot\vf_k(\by),  
\end{equation}
where the positive functions  $\mathcal{D}_k(z)$ are defined by \eqref{e:fundamental-solution-bessel-type} and \eqref{e:generic-fundamental-solutions} depending upon the Fourier modes which solve
\begin{align}
\frac{d^2 \mathcal{D}_k(z)}{dz^2}-\Big(\frac{j_k^2n^2}{4}\cdot z^{n}+n\lambda_k\Big)\cdot z^{n-2}\cdot \mathcal{D}_k(z)=0,\ z\geq 1,\label{e:D_k-equation}\end{align}
  where $\lambda_k\geq 0$ for every $k\in\dZ_+$ (see \eqref{e:lambda_k-lower-bound}). Moreover, each $\mathcal{D}_k(z)$ yields some definite exponentially decaying rate (see Lemma \ref{l:j=0-grow-asymp} and Lemma \ref{l:generic-asymp} for the accurate rates). 

First, we prove $\kappa = \kappa_0$. 
In fact, 
\begin{equation}
\frac{\p u(z,\by)}{\p z} = \kappa + \sum\limits_{k=1}^{\infty} c_k \cdot \mathcal{D}_k'(z)\cdot\vf_k(\by).  \label{e:partial-derivative}
\end{equation}
Integrating \eqref{e:partial-derivative} over $(Y^{2n-1},h_0)$ and evaluating at $z=z_0$, 
\begin{equation}
\kappa\cdot\Vol_{h_0}(Y^{2n-1})=\int_{Y^{2n-1}}\Big(\frac{\p u(z,\by)}{\p z}\Big|_{z=z_0}\Big)\dvol_{h_0} = \kappa_0\cdot \Vol_{h_0}(Y^{2n-1}),
\end{equation}
which implies 
$\kappa=\kappa_0$.

Next, we prove $c_k=0$ for all $k\in\dZ_+$. In fact, for each fixed $k\in\dZ_+$, multiplying $\vf_k$ on the both sides of \eqref{e:partial-derivative} and integrating over $Y^{2n-1}$,
\begin{equation}  c_k \cdot \mathcal{D}_k'(z)=\int_{Y^{2n-1}}\vf_k(\by)\cdot \frac{\p u(z,\by)}{\p z}\Big|_{z=z_0}=0. 
\end{equation}
Then the conclusion $c_k=0$ for every $k\in\dZ_+$ follows from the claim 
\begin{equation}
\mathcal{D}_k'(z)<0 \ \text{for all} \ z \geq z_0.
\end{equation}
Now we just need to prove the claim.
Since $\mathcal{D}_k$ satisfies \eqref{e:D_k-equation} and noticing
 $\lambda_k>0$ for every $k\in\dZ_+$, we have that 
$\mathcal{D}_k''(z)>0$ in $[z_0,+\infty)$. Then $\mathcal{D}_k'(z)$ is increasing in $[z_0,+\infty)$. The decay $\lim\limits_{z\to+\infty}\mathcal{D}_k(z)= 0$ implies $\lim\limits_{z\to+\infty}\mathcal{D}_k'(z)=0$, and hence $\mathcal{D}_k'(z)<0$ in $[z_0,+\infty)$.

The above arguments imply that  $u(z,\by)\equiv\kappa_0\cdot z + \ell_0$ on $\Ca$.  The proof is done. 
\end{proof}

\begin{corollary}[Dirichlet boundary]
\label{c:Liouville-dirichlet}
In the above notations, let $u$ be a solution of the Dirichlet boundary problem on the Calabi space $(\Ca, g_{\Ca})$ based over a compact Calabi-Yau manifold $(D,\omega_D)$,
\begin{align}
\begin{cases}
\Delta_{g_{\Ca}} u(\bx) = 0, & \bx\in\Ca,
\\
u(\bx) = 0, & z(\bx)= z_0.
\end{cases}
\end{align}
If $u$ satisfies the growth condition $|u(\bx)| = O(e^{\delta\cdot z(\bx)^{\frac{n}{2}}})$ for some  $0<\delta<2(\frac{\lambda_D}{n})^{\frac{1}{2}}$, 
then $u$ must be a linear function on $\Ca$, i.e., there exists a constant $\kappa_0\in\dR$ such that  $u = \kappa_0 \cdot (z-z_0)$ in terms of the natural moment map coordinate $z$ on $\Ca$.

\end{corollary}

\begin{proof}
The proof follows quickly from Proposition \ref{p:harmonic-function-decay}, and the exponentially decaying terms are vanishing due to the Dirichlet boundary condition, which is similar to the proof of Corollary \ref{c:Liouville-neumann}. So we omit the details.
\end{proof}

\section{The Poisson equation with prescribed asymptotics}
\label{s:Poisson with asymptotics}

In this subsection, we will construct solutions to the Poisson equation on the Calabi space $(\Ca, g_{\Ca})$,
\begin{equation}
\Delta_{g_{\Ca}} u = v
\end{equation}
with {\it controlled asymptotic behavior}.
As in Section \ref{s:ode-setup}, we carry out separation of variables. Suppose $v$ is a smooth function defined on $\{z\geq 1\}$. We write 
\begin{align}
u(z,\by) = \sum\limits_{k=1}^{\infty} u_k (z) \cdot \vf_k (\by),
\quad 
 v(z, \by) = \sum\limits_{k=1}^{\infty} \xi_k (z) \cdot \vf_k (\by).
\end{align}
So the Poisson equation
\begin{equation}
\Delta_{g_{\Ca}} u = v
\end{equation}
is reduced to the following inhomogeneous ODE
\begin{equation}
\frac{d^2 u_k(z)}{dz^2}-(\frac{j_k^2n^2}{4}\cdot z^{n}+n\lambda_k)z^{n-2}u_k(z)=z^{n-1}\cdot \xi_k(z),\quad  z\geq z_1.\label{e:poisson-ode}\end{equation}
Let $\mathcal{G}_k(z)$ and $\mathcal{D}_k(z)$ be 
the growing solution and decaying solution to the corresponding 
homogeneous equation, which were analyzed in Section \ref{s:j=0} and \ref{s:jk not zero}.
 So applying standard Liouville' formula, Equation \eqref{e:poisson-ode} 
has a particular solution

\begin{equation}
u_{k}(z)\equiv  \frac{\mathcal{G}_k(z)}{\mathcal{W}_k(z)}\int_z^{\infty}\mathcal{D}_k(r)\cdot\Big(\xi_{k}(r)\cdot r^{n-1}\Big)dr+\frac{\mathcal{D}_k(z)}{\mathcal{W}_k(z)}\int_{z_1}^{z}\mathcal{G}_k(r)\cdot\Big(\xi_{k}(r)\cdot r^{n-1}\Big)dr , \label{e:particular-solution}\end{equation}
where $\mathcal{W}_k$ is the Wronskian
\begin{equation}
\mathcal{W}_k(z) \equiv \mathcal{W}\Big(\mathcal{G}_k(z),\mathcal{D}_k(z)\Big).\end{equation}

\begin{lemma}\label{l:coefficients-estimate}

Assume that the function $\xi_k(z)$ satisfies the following property:
 there are $\eta_0\in(-\delta_b/2,\delta_b/2)$, a sequence of positive constants $\mathfrak{B}_k>0$ such that 
\begin{equation}
|\xi_{k}(z)| \leq \mathfrak{B}_k \cdot  e^{\eta_0\cdot z^{\frac{n}{2}}}.\label{e:xi-bound}
\end{equation}
Let $u_{k}(z)$ be the particular solution \eqref{e:particular-solution}, then there exists some constant $C_0>0$ such that the particular solution $u_k$ satisfies the uniform estimate 
\begin{equation}
|u_{k}(z)| \leq C_0 \cdot \mathfrak{B}_k\cdot (\Lambda_k)^{\frac{1}{2n}}\cdot e^{\eta\cdot z^{\frac{n}{2}}} \end{equation}
 for any $\eta>\eta_0$.
\end{lemma}

\begin{proof}

We will estimate the two terms in \eqref{e:particular-solution} individually, and we also divide into several cases.

First consider $j_k=0$ and $k=0$. In this case the solutions $u_k$ is given by simple integrals of $\xi_k$ and the conclusion is easy to see. 

The second case is that $k\in\dZ_+$ and $j_k=0$. Applying Proposition \ref{p:bessel-functions-estimate}, the fundamental solutions $\mathcal{G}_k(z)$ and $\mathcal{D}_k(z)$ satisfy the uniform estimates
\begin{align}
\mathcal{G}_k(z) & \leq \frac{C}{\lambda_k^{\frac{1}{4}}}\cdot z^{\frac{2-n}{4}} \cdot e^{2\sqrt{\frac{\lambda_k}{n}}\cdot z^{\frac{n}{2}}},
\\
\mathcal{D}_k(z) & \leq  \frac{C}{\lambda_k^{\frac{1}{4}}}\cdot z^{\frac{2-n}{4}} \cdot e^{-2\sqrt{\frac{\lambda_k}{n}}\cdot z^{\frac{n}{2}}}. 
\end{align}
By Lemma \ref{l:j=0-Wronskian},  $\mathcal{W}_k(z)=\mathcal{W}(\mathcal{G}_k(z),\mathcal{D}_k(z))=\frac{n}{2}$. 
Let us denote $\tilde{\lambda}_k\equiv2\sqrt{\frac{\lambda_k}{n}}$, then $\tilde{\lambda}_k\geq 2\sqrt{\frac{\lambda_1}{n}} = \delta_b$. Now the first integral term in \eqref{e:particular-solution} has the following bound, 
\begin{eqnarray}
&& \frac{\mathcal{G}_k(z)}{\mathcal{W}_k(z)} \int_{z}^{\infty} \mathcal{D}_k(r) |\xi_k(r)\cdot r^{n-1}|dr 
\nonumber\\
&\leq& \frac{C\cdot \mathfrak{B}_k}{\lambda_k^{\frac{1}{2}}}\cdot z^{\frac{2-n}{4}}\cdot e^{\tlk\cdot z^{\frac{n}{2}}}\cdot \int_{z}^{\infty} r^{\frac{3n}{4}-\frac{1}{2}}\cdot e^{(-\tlk+\eta_0)\cdot r^{\frac{n}{2}}}dr.
\end{eqnarray}
By assumption, $|\eta_0|< \frac{\delta_b}{2} \leq \frac{\tilde{\lambda}_k}{2} $, then 
\begin{eqnarray}
\nonumber\\
\frac{\mathcal{G}_k(z)}{\mathcal{W}_k(z)} \int_{z}^{\infty} \mathcal{D}_k(r) |\xi_k(r)\cdot r^{n-1}|dr 
 &\leq& \frac{C\cdot \mathfrak{B}_k}{\lambda_k^{\frac{1}{2}}}\cdot z^{\frac{2-n}{4}}\cdot e^{\tlk\cdot z^{\frac{n}{2}}}\cdot e^{(-\tlk+\eta')\cdot z^{\frac{n}{2}}} 
\nonumber\\
&\leq & C \cdot \mathfrak{B}_k\cdot e^{\eta\cdot z^{\frac{n}{2}}}, \end{eqnarray}
where $\eta>\eta'>\eta_0>0$.
Similarly, 
\begin{eqnarray}
\frac{\mathcal{D}_k(z)}{\mathcal{W}_k(z)} \int_{z_0}^{z} \mathcal{G}_k(r) |\xi_k(r)\cdot r^{n-1}|dr
\leq   C \cdot \mathfrak{B}_k \cdot  e^{\eta\cdot z^{\frac{n}{2}}}.
\end{eqnarray}

In the third case $j_k\in\dZ_+$ and  $Q\geq 1$, we need to apply Lemma \ref{l:monotonicity}. In fact, 
 \begin{eqnarray}
&&\frac{\mathcal{G}_k(z)}{\mathcal{W}_k(z)} \int_{z}^{\infty} \mathcal{D}_k(r) \xi_k(r)\cdot r^{n-1}dr
 \nonumber \\
&\leq&C_n\cdot \frac{Q^{\frac{1}{4}}\cdot(j_k\cdot z^n)^{\frac{1-2\alpha}{4}}}{\Gamma^2(Q+1)}\cdot \frac{e^{\widehat{G}_k(z)}}{\mathcal{W}_k(z)}\int_z^{\infty}e^{\widehat{F}_k(r)} \xi_k(r)\cdot r^{n-1} dr
\nonumber\\
&\leq& C_n\cdot\mathfrak{B}_k\cdot\frac{Q^{\frac{1}{4}}\cdot(j_k\cdot z^n)^{\frac{1-2\alpha}{4}}}{\Gamma^2(Q+1)}\cdot\frac{e^{\widehat{G}_k(z)}}{\mathcal{W}_k(z)}\int_z^{\infty}e^{\widehat{F}_k(r)+\eta' \cdot r^{\frac{n}{2}}} dr,
\end{eqnarray}
where $\eta'>\eta_0$. We choose  any $\epsilon\in( \delta_b/100,\delta_b/10)$ and denote $\eta'' \equiv \eta' + \epsilon$, then
by Lemma  \ref{l:monotonicity},
\begin{eqnarray}
&&\frac{e^{\widehat{G}_k(z)}}{\mathcal{W}_k(z)}\int_z^{\infty}e^{\widehat{F}_k(r)+\eta' \cdot r^{\frac{n}{2}}} dr \nonumber \\
&= &\frac{e^{\widehat{G}_k(z)}}{\mathcal{W}_k(z)}\int_z^{\infty}e^{\widehat{F}_k(r)+\eta'' \cdot r^{\frac{n}{2}}} \cdot e^{-\epsilon r^{\frac{n}{2}}}dr
\nonumber\\
& \leq & \frac{e^{\widehat{F}_k(z)+\widehat{G}_k(z)+\eta ''\cdot z^{\frac{n}{2}}}}{  \mathcal{W}_k(z)}\int_z^{\infty}e^{-\epsilon \cdot r^{\frac{n}{2}}}dr 
\nonumber \\
& \leq &C_n\cdot\frac{e^{\widehat{F}_k(z)+\widehat{G}_k(z)+\eta'' \cdot z^{\frac{n}{2}}}}{\mathcal{W}_k(z)}.
\end{eqnarray}
Therefore, 
\begin{equation}
\frac{\mathcal{G}_k(z)}{\mathcal{W}_k(z)} \int_{z}^{\infty} \mathcal{D}_k(r) \xi_k(r)\cdot r^{n-1}dr
\leq C_n\cdot\mathfrak{B}_k\cdot\frac{Q^{\frac{1}{4}}\cdot (j_k\cdot z^n)^{\frac{1-2\alpha}{4}}}{\Gamma^2(Q+1)}\cdot\frac{e^{\widehat{F}_k(z)+\widehat{G}_k(z)+\eta''\cdot z^\frac{n}{2}}}{\mathcal{W}_k(z)}.
\end{equation}
Plugging Lemma \ref{l:product-estimate} and Proposition \ref{p:generic-Wronskian} into the above inequality, 
\begin{eqnarray}
\frac{\mathcal{G}_k(z)}{\mathcal{W}_k(z)} \int_{z}^{\infty} \mathcal{D}_k(r) \xi_k(r)\cdot r^{n-1}dr
&\leq& C_n \cdot \mathfrak{B}_k \cdot \frac{ j_k^{\frac{1}{n}}\cdot e^{-Q}\cdot Q^{Q+1}}{\Gamma(Q+1)}\cdot z \cdot  e^{\eta''\cdot z^{\frac{n}{2}}}.
\nonumber\\
& \leq&  C_n\cdot \mathfrak{B}_k\cdot j_k^{\frac{1}{n}}\cdot e^{\eta \cdot z^{\frac{n}{2}}}
\nonumber\\
&\leq &  C_n\cdot \mathfrak{B}_k\cdot (\Lambda_k)^{\frac{1}{2n}}\cdot e^{\eta \cdot z^{\frac{n}{2}}}
\end{eqnarray}
for any $\eta\in(\eta'',\eta''+\frac{\delta_b}{100})$, where we used Stirling's formula for estimating $\Gamma(Q+1)$.
Similarly we get the bound for the other term of \eqref{e:particular-solution}. 

The fourth case is when $j_k\geq 1$ and $Q\leq 1$. This case is simpler and follows from Corollary \ref{c:Dk upper and lower bound Q small case} and the argument in the second case.

This completes the proof of the proposition. 
\end{proof}

Based on the above ODE estimate, we 
prove the following $C^0$ and $C^1$ estimate for the equation to the Poisson equation.

\begin{proposition}\label{p:poisson-solvability}Let $\{z\geq 1\}\subset \Ca$ be a subset and let  $K_0\geq 2n+1$ be a positive integer. 
 Given any $\eta_0 \in(-\delta_b/2,\delta_b/2)\setminus
 \{0\}$,  if $v\in C^{3K_0,\alpha}(\{z\geq 1\})$ for
and 
\begin{equation}|v|=O(e^{\eta_0 \cdot z(\bx)^{\frac{n}{2}}}),\end{equation}
 then the Poisson equation
\begin{equation}
\Delta_{g_{\Ca}}u=v\label{e:possion-eq}
\end{equation}
has a solution $u\in C^{3K_0+2,\alpha}(\{z\geq 1\})$ such that for any $\eta>\eta_0$
\begin{equation}
|u(\bx)|+|\nabla_{g_{\mathcal{C}^n}} u(\bx)|_{\mathcal{C}^n}\leq C\cdot e^{\eta \cdot z^{\frac{n}{2}}},\label{e:poisson-solution-estimate}
\end{equation}
as $z(\bx)\to+\infty$, where $C>0$ is independent of $\bx\in\Ca$.

\end{proposition}

\begin{proof}

The proof is constructive, which will be done in two steps. 

The first step, as the main part, is to find a solution $u$ with the prescribed growth (or decay) rate. 
We will use the method of separation of variables described as follows.

For a fixed slice $Y^{2n-1}\subset \Ca$, let $\{\Lambda_k\}_{k=0}^{\infty}$ with $\Lambda_0=0$ be the spectrum of $\Delta_{\Ca}$ acting on functions. Let $\{\varphi_k\}_{k=0}^{\infty}$ be the eigenfunctions satisfying
\begin{align}
\begin{cases}
-\Delta_{\Ca}\varphi_k=\Lambda_k\varphi_k,
\\
\|\varphi_k\|_{L^2(Y^{2n-1})}=1.
\end{cases}
\end{align}
Given a function $v$ and for any fixed $z\geq 1$, we have the fiberwise $L^2$-expansion on $Y^{2n-1}$,
\begin{equation}
v(z,\by) = \sum\limits_{k=1}^{\infty} v_k(z) \vf_k(\by).
\end{equation}
Then we can first construct a formal solution
\begin{equation}
u(z,\by)=\sum\limits_{k=1}^{\infty}u_{k}(z)\vf_k(\by)
\end{equation}
to \eqref{e:possion-eq}, which holds in the $L^2$-sense for each fixed $z\geq 1$. 
Here the coefficient functions $u_{k}(z)$ are the particular solutions constructed in Lemma \ref{l:coefficients-estimate}. 
The main part is to prove that the above series $u(z,\by)$ converges with higher regularity and hence $u(z,\by)$ is a regular solution to \eqref{e:possion-eq}.

To begin with, we will prove that the series $u(z,\by)$ converges in the $C^0$-norm and hence gives  a $C^0$-function. Combining Lemma \ref{l:error-estimate}, Lemma \ref{l:coefficients-estimate} and the eigenfunction estimate  in Lemma \ref{l:eigenfunction-bound}, we have 
\begin{align}
|u(z,\by)| 
\leq  \sum\limits_{k=1}^{\infty}|u_{k}(z)| \cdot |\vf_k(\by)| 
\leq C\sum\limits_{k=1}^{\infty}\frac{e^{\eta \cdot z^{\frac{n}{2}}}}{(\Lambda_k)^{K_0-\frac{n}{2}-\frac{1}{2n}}}. 
\label{e:numerical-series}
\end{align}
 Applying Weyl's 
law to the spectrum  $\{\Lambda_k\}_{k=1}^{\infty}$,
\begin{equation}
C_0^{-1} k^{\frac{2}{2n-1}}\leq|\Lambda_k| \leq C_0 k^{\frac{2}{2n-1}},
\end{equation}
where $C_0>0$ depends only on  $Y^{2n-1}$ and $k$ is sufficiently large.
Let $K_0\geq 2n+1$, then \begin{equation}
|u(z,\by)|\leq C\cdot e^{\eta\cdot z^{\frac{n}{2}}}\cdot \sum\limits_{k=1}^{\infty}\frac{1}{(\Lambda_k)^{\frac{3n}{2}}} \leq C\cdot e^{\eta\cdot z^{\frac{n}{2}}}\cdot\sum\limits_{k=1}^{\infty} \frac{1}{k^{\frac{3n}{2n-1}}} \leq  C\cdot e^{\eta\cdot z^{\frac{n}{2}}}.
\end{equation}
Therefore, 
$u\in C^0(\Ca)$ and $u$ satisfies the $C^0$-asymptotic estimate in \eqref{e:poisson-solution-estimate}. 

Based on the above $C^0$-regularity, we will apply the standard elliptic regularity on $(\Ca,g_{\Ca})$ to show that $u\in C^2(\Ca)$ is a regular solution to $\Delta_{g_{\Ca}} u =v$.
We take the partial sums
\begin{align}
U_N(z,\by)\equiv\sum\limits_{k=1}^{N}u_k(z)\vf_k(\by), \
V_N(z,\by)\equiv\sum\limits_{k=1}^{N}v_k(z)\vf_k(\by)
\end{align}
of the expansions
\begin{align}
u(z,\by)=\sum\limits_{k=1}^{\infty}u_k(z)\vf_k(\by), \
v(z,\by)=\sum\limits_{k=1}^{\infty}v_k(z)\vf_k(\by).
\end{align}
It is obvious that,
\begin{equation}
\Delta_{g_{\Ca}}U_N=V_N.
\end{equation}
For every $\bx\equiv (z,\by)\in \Ca$, we will apply the elliptic regularity on the ball $B_2(\bx)\subset\Ca$ to obtain the higher regularity of $u$.

As a starter, by the same arguments as the above, we have
$\|V_N-v\|_{C^0(B_2(\bx))}\to 0$
as $N\to\infty$.
The proof of the higher order convergence is almost verbatim. In fact, we just need to use $\|v\|_{C^{2K_0+m}}$ with $m\leq K_0$.
Since $\Delta_{g_{\mathcal{C}}} U_N=V_N$, the  standard $W^{2,p}$- implies that regularity 
for every $1<p<\infty$, 
\begin{equation}\|U_N\|_{W^{2,p}(B_1(\bx))}\leq C_{p,\bx}\cdot (\|V_N\|_{C^0(B_2(\bx))}+(\|U_N\|_{C^0(B_2(\bx))}).\end{equation} By assumption $v\in C^{3K_0}(\Ca)$ for  $K_0\geq 2n+1$, so it follows that 
$\|V_N\|_{C^2(B_2(\bx))}\leq C_{\bx}$. Therefore, for every $1<p<\infty$, 
\begin{equation}
\|U_N\|_{W^{4,p}(B_1(\bx))}\leq C_{p,\bx}  ( \|U_N\|_{W^{2,p}(B_{3/2}(\bx))} + \|V_N\|_{W^{2,p}(B_2(\bx))} )\leq C_{p,\bx}.
\end{equation}
Now it suffices to choose $p>2n$,  so the Sobolev embedding implies  
\begin{equation}
\|U_N\|_{C^{3,\alpha}(B_1(\bx))}\leq C_{p,\bx},\ \alpha \equiv 1-\frac{2n}{p},
\end{equation}
which implies that $U_N\to u$ in the $C^3$-norm with respect to $g_{\Ca}$.  The proof of the first step 
is done.

We have constructed a solution $u$ satisfying $|u(\bx)| \leq C\cdot e^{\eta\cdot z^{\frac{n}{2}}}$.
Now we are ready to show that 
\begin{equation}
|\nabla u(\bx)| \leq C\cdot e^{\eta\cdot z^{\frac{n}{2}}}.\label{e:gradient-u-decay}
\end{equation}
This can be accomplished by the elliptic $W^{2,p}$-estimate. Since a Calabi space
$(\Ca,g_{\Ca})$ is collapsed with bounded curvatures as $z\to+\infty$, so there is some constant $r_0>0$ such that for each $\bx\in\Ca$ satisfying $z(\bx)\geq 1$,  the universal cover $(\widetilde{B_{2r_0}(\bx)},\tilde{\bx})$ is non-collapsing. Now we lift the solution $u$ to this non-collapsing local universal cover, then for any $p>1$, there exists $C_p>0$ such that
 \begin{equation}
|u|_{W^{2,p}(B_{r_0}(\tilde{\bx}))} \leq C_p\cdot (|u|_{L^{\infty}(B_{2r_0}(\tilde{\bx}))} + |v|_{L^{\infty}(B_{2r_0}(\tilde{\bx}))})\leq C_{p}\cdot e^{\eta\cdot z^{\frac{n}{2}}}.
\end{equation}
 We can choose any $p>2n$, then Sobolev embedding gives
 \begin{equation}
 |u|_{C^{1,\alpha}(B_{r_0}(\tilde{\bx}))}  \leq C\cdot e^{\eta\cdot z^{\frac{n}{2}}}.
 \end{equation}
In particular, 
\begin{equation}
|\nabla u(\bx)| \leq C\cdot e^{\eta\cdot z^{\frac{n}{2}}},
\end{equation}
where $\alpha\equiv 1-\frac{2n}{p}$.  
So the proof of the proposition is done. 
 \end{proof}

\section{Proof of the Liouville theorem}
\label{s:proof of Liouville}
In this subsection, we will complete the proof of Theorem \ref{t:liouville-theorem-calabi}.

To begin with, we prove the following lemma, which states that any harmonic function with slow exponential growth rate on a $\delta$-asymptotically Calabi space is in fact {\it almost harmonic} with repsect to the Calabi model metric. 

 \begin{lemma}\label{l:almost-harmonic}
Let $(X^{2n},g)$ be a complete non-compact Riemannian manifold which is 
$\delta$-asymptotically Calabi space in the sense of Definition \ref{d:asymptotic-Calabi}. Let $\ldel>0$ be a  constant 
such that if $u$ satisfies
\begin{align}
\Delta_{g} u  = 0 \quad \text{and}\quad 
u = O(e^{\ldel \cdot z^{\frac{n}{2}}}),
\end{align}
then there exists $z_0>0$, such that for every fixed $k\in \dZ_+$, we have for all $z\geq z_0$,
 \begin{equation}
 |\nabla^k_{g_{\mathcal C^n}} \Delta_{g_{\Ca}} u(z, \by)|_{g_{\mathcal{C}^n}} \leq C_k \cdot e^{(\ldel-\delta)\cdot z^{\frac{n}{2}}},
 \end{equation}
 where $C_k$ is a constant depending only on $X$ and $k$. 
\end{lemma}
The proof of this is essentially the same as the proof of Claim 4.18 in \cite{HSVZ}. We omit the details here. By quite explicit computations,  the curvatures of the Calabi model space are uniformly bounded as $z\to+\infty$, which allows us to use the local elliptic estimate even though the geometry is collapsing at infinity.

\begin{proof}
[Proof of Theorem \ref{t:liouville-theorem-calabi}]
For the given $\delta$-asymptotically Calabi space $(X^{2n,g})$, let $(\Ca, g_{\Ca})$ be the incomplete Calabi model space with the associated divisor $D$.
Let us denote 
\begin{equation}
\epsilon_X\equiv \min\{\delta, \delta_b\},
\end{equation} where $\delta_b\equiv 2\left(\frac{\lambda_D}{n}\right)^{\frac{1}{2}}>0$ is the constant defined by \eqref{e:definition of deltab} in Proposition \ref{p:harmonic-function-decay}.

Let $u$ be a harmonic function on the $\delta$-asymptotically Calabi space $(X^{2n}, g)$ such that for some $\epsilon\in(0,\ty)$, $u$ satisfies the growth condition
\begin{equation}
u = O(e^{ \epsilon\cdot  z^{\frac{n}{2}}})\label{e: exp-growth}
\end{equation}
By assumption, there exists some large constant $z_1\gg1$, and a diffeomorphism
\begin{equation}
\Phi: [z_1,+\infty)\times Y^{2n-1} \to X^{2n}\setminus K \end{equation}
such that for all $k\in\dN$
\begin{equation}
|\nabla_{g_{\Ca}}^k(\Phi^*g-g_{\Ca})|_{g_{\Ca}}  \leq C e^{-\delta\cdot z^{\frac{n}{2}}}.\label{e:C1-exp-close}
\end{equation}
By the Lemma \ref{l:almost-harmonic}, there is some large constant $z_0\gg1$ such that 
\begin{align}
\Delta_{g_{\mathcal C^n}}u &=\phi,
\\
|\nabla^k_{g_{\mathcal C^n}}\phi|_{g_{\Ca}}&=O(e^{(\epsilon-\delta)\cdot z^{\frac{n}{2}}})
\end{align}
for all $z\geq z_0$ and $k\in\dN$, where $\epsilon<\ty\leq\delta$. 

Then applying Proposition \ref{p:poisson-solvability}  on $[z_0,+\infty)\times Y^{2n-1}$, there exists a solution to the equation
\begin{equation}
\Delta_{g_{\Ca}}v=\phi\label{e:surjective}
\end{equation}
such that 
\begin{equation}|v|+|\nabla_{g_{\Ca}} v|_{g_{\Ca}}=O(e^{-\ell\cdot  z^{\frac{n}{2}}})\end{equation} for any $\ell\in(0,\delta-\epsilon)$.
Notice that, as $z\to+\infty$, curvatures are uniformly bounded in the Calabi space. 
Therefore, we have
\begin{align}
0 = \Delta_{g}u = \Delta_{g_{\Ca}} (u - v), \label{e:model-harmonic}
\end{align}
and $u - v = O( e^{\epsilon \cdot z^{\frac{n}{2}}})$. 
Since $\epsilon<\ty\leq\delta_b$, now we are in a position to apply Proposition \ref{p:harmonic-function-decay} to $u - v$, which shows that there is some harmonic function $h$ on the Calabi space such that
\begin{equation}u - v = \kappa_0\cdot  z + c_0 +h,\end{equation}
where $|h|+|\nabla_{g_{\Ca}} h|_{g_{\Ca}}=O(e^{-\underline{\delta}\cdot z^{\frac{n}{2}}})$ 
for all $\ldel\in (0, \delta_b)$. 
Also $|dz|_{g_{\mathcal C^n}}\rightarrow 0$ as $z\rightarrow\infty$, then 
\begin{equation}
|du|_g\leq C|du|_{g_{\mathcal C^n}} \leq C(|dv|_{g_{\mathcal C^n}}+|dz|_{g_{\mathcal C^n}}+|dh|_{g_{\mathcal C^n}})\rightarrow 0,  \ \ \ \ z\rightarrow\infty. \label{e:du-decay-estimate}
\end{equation}
Since $\Delta_{g}u=0$, so it holds that
 \begin{equation}
 \Delta_H(du)=dd^*(du)=-d\Delta_{g}u=0,
 \end{equation}
 where $\Delta_H$ is the Hodge Laplacian on $(X^{2n}, g)$. 
By assumption, $(X^{2n},g)$ satisfies $\Ric_{g}\geq 0$, then Bochner's formula implies that  \begin{equation}
\frac{1}{2}\Delta_g|du|_g^2 =  |\nabla_g du|_g^2
+  \Ric_g(du, du) \geq 0.\end{equation}
Applying the decay property of $|du|$ in \eqref{e:du-decay-estimate} and the maximum principle, \begin{equation}
|du|_{g}\equiv0 \ \text{on}\ X^{2n}.
\end{equation} 
Therefore, $u$
is a constant.
\end{proof}

\appendix

\section{Some formulae in special functions}

\label{s:appendix-1}

For developing quantitative estimates in this paper, we need to use some formulae
and facts about the modified Bessel functions and the confluent hypergeometric functions.  
Some formulae applied in our concrete setting are in fact not completely standard in the literature, which deserves some proof. 
For making the paper the self-contained and for readers' convenience, 
we try to summarize those results with detailed and checkable proofs in this section. Our main reference  is \cite{Lebedev}.

\subsection{Modified Bessel functions}

Let $\nu\in\dR$, we consider the following {\it modified Bessel equation}  
\begin{equation}
y^2\cdot\frac{d^2 \mathcal{B}(y)}{dy^2}+y\cdot\frac{d\mathcal{B}(y)}{dy}-(y^2+\nu^2)\cdot \mathcal{B}(y)=0,\ y\geq 0.\label{e:m-b-eq}
\end{equation}
First, for any $\nu\in\dR$, we define
\begin{align}
I_{\nu}(y)\equiv\sum\limits_{k=0}^{\infty}\frac{1}{\Gamma(k+1)\Gamma(k+\nu+1)}\Big(\frac{y}{2}\Big)^{2k+\nu}.\end{align}
In the special case $\nu=-\ell$ with $\ell\in\dZ_+$, then the above definition can be also explained as 
\begin{equation}
I_{\nu}(y)=  \sum\limits_{k=\ell}^{\infty}\frac{1}{\Gamma(k+1)\Gamma(k-\ell+1)}\Big(\frac{y}{2}\Big)^{2k-\ell}.
\end{equation}
Immediately, for any positive integer $\ell\in\dZ_+$, we have
\begin{equation}
I_{-\ell}(z) = I_{\ell}(z).
\end{equation}
Next we define $K_{\nu}(z)$
as follows,
\begin{align}
K_{\nu}(y)\equiv \begin{cases}\frac{\pi}{2\sin(\nu\pi)}\cdot (I_{-\nu}(y)-I_{\nu}(y)), & \nu\not\in\dZ,
\\
\lim\limits_{\substack{\nu'\to\nu\\\nu'\not\in\dZ}}K_{\nu'}(y),  & \nu\in\dZ.
 \end{cases}\label{e:def-K-function}
\end{align}
One can check that $I_{\nu}(y)$ and $K_{\nu}(y)$ are two linearly independent solutions to \eqref{e:m-b-eq}. In the literature,  $I_{\nu}$  and $K_{\nu}$ are usually called {\it modified Bessel functions}.

In our context, mainly we are interested in the solutions $I_{\nu}$ and $K_{\nu}$ with an index $\nu=\frac{1}{n}$ and $n\geq 2$. 
The simples case is $n=2$
such that both $I_{\frac{1}{2}}(y)$ and $K_{\frac{1}{2}}(y)$ have explicit formulae:
\begin{equation}
I_{\frac{1}{2}}(y) =\sqrt{\frac{2}{\pi y}} \sinh(y),\ K_{\frac{1}{2}}(y) =\sqrt{\frac{\pi}{2 y}} e^{-y}.
\end{equation}

The main part of this subsection is to prove the following useful integral representations for $I_{\nu}$ and $K_{\nu}$.

\begin{lemma}\label{l:infinite-integral}
Given $\nu\in\dR$, 
then the following integral formulae hold for each $y>0$, 
\begin{align}
I_{\nu}(y)&=\frac{1}{\pi}\int_{0}^{\pi}e^{y\cos\theta} \cos(\nu\theta) d\theta - \frac{\sin(\nu\pi)}{\pi}\int_0^{\infty}e^{-y\cosh t - \nu t} dt,
\\
K_{\nu}(y) &= \int_{0}^{\infty}e^{-y\cosh t}\cosh(\nu t)dt.
\end{align}

\end{lemma}

\begin{proof}
First, we prove the integral formula for $I_{\nu}$. The idea of the proof was originally inspired by Hankel's  representation formula for the reciprocal gamma function. In fact, let $\mathcal{L}\subset \dC$ be a contour winding around the negative $Ox$-axis. In our particular case, $\mathcal{L}=\mathcal{L}_1 + \mathcal{L}_2+\mathcal{L}_3$, where $\mathcal{L}_1$ and $\mathcal{L}_3$ are two rays parallel to $Ox$ and $\mathcal{L}_2$ is an arc of the unit circle centered at the origin (See Figure \ref{f:hh}). So Hankel's representation formula gives that
\begin{equation}
\frac{1}{\Gamma(k+\nu+1)}=\frac{1}{2\pi\sqrt{-1}}\int_{\mathcal{L}}e^{w} w^{-(k+\nu+1)}dw,\ w\in\dC.\label{e:1/Gamma}
\end{equation}
By the power series definition of $I_{\nu}$,
\begin{eqnarray}
I_{\nu}(y)&=&\sum\limits_{k=0}^{\infty}\frac{1}{\Gamma(k+1)\Gamma(k+\nu+1)}\Big(\frac{y}{2}\Big)^{2k+\nu}
\nonumber\\
&=&(\frac{y}{2})^{\nu}\frac{1}{2\pi\sqrt{-1}}\int_{\mathcal{L}}e^w w^{-\nu-1}\sum\limits_{k=0}^{\infty}\frac{(\frac{y^2}{4w})^{k}}{k!}dw
\nonumber\\
&=&(\frac{y}{2})^{\nu}\frac{1}{2\pi\sqrt{-1}}\int_{\mathcal{L}}e^{w+\frac{y^2}{4w}} w^{-\nu-1}dw.
\end{eqnarray}
For every $y>0$, we make change of variables for each $w\in\dC$, \begin{equation}w=\frac{y\cdot e^{\zeta}}{2}=\frac{ye^t}{2} \cdot e^{\sqrt{-1}\theta},\ 0< t<\infty,\ 0\leq \theta\leq 2\pi.\end{equation} 
Letting $\mathcal{L}_1$ and $\mathcal{L}_3$ tend to each other, then in terms of the variables $(t,\theta)$,
\begin{equation}
\int_{\mathcal{L}} e^{w+\frac{y^2}{4w}} w^{-\nu-1}dw
=\frac{1}{\pi}\int_{0}^{\pi}e^{y\cos\theta} \cos(\nu\theta) d\theta - \frac{\sin(\nu\pi)}{\pi}\int_0^{\infty}e^{-y\cosh t - \nu t} dt.
\end{equation}

The integral formula for $K_{\nu}$
follows easily from the above integral representation for $I_{\nu}$ and the definition
\begin{equation}
K_{\nu}(y)=\frac{\pi(I_{-\nu}(y)-I_{\nu}(y))}{2\sin(\nu\pi)}.\end{equation}

 \begin{figure}
\begin{tikzpicture}[scale = 0.5]

\draw (-5,0) -- (4,0);

\draw (0,-3) -- (0,3);

\draw (0,3) -- (0.1, 2.8);

\draw (0,3) -- (-0.1, 2.8);

\draw (4,0) -- (3.8, -0.1) ; 

\draw (4,0) -- (3.8, 0.1) ;

\draw[thick] (1.41,1.41) -- (1.67,1.35);

\draw[thick] (1.41,1.41) -- (1.50,1.15);

\draw[very thick] (-1.97,-0.35) arc (-170:170:2);

\draw[very thick] (-1.94,-0.35) -- (-4.8,-0.35);

\draw[very thick] (-1.94,0.35) -- (-4.8,0.35);

\draw[thick] (-4,0.35) -- (-3.7,0.50);

\draw[thick] (-4,0.35) -- (-3.7,0.20);

\draw[thick] (-3.7,-0.35) -- (-4.0,-0.20);

\draw[thick] (-3.7,-0.35) -- (-4.0,-0.50);

\node at (1.9,1.9) {$\mathcal{L}_2$};

\node at (0.40, -0.5) {$O$};

\node at (4.0, -0.4) {$x$};

\node at (-0.3, 2.8) {$y$};

\node at (-4.0, -1.0) {$\mathcal{L}_1$};
\node at (-4.0, 1.0) {$\mathcal{L}_3$};

\end{tikzpicture}
\caption{The contour $\mathcal{L}=\mathcal{L}_1+\mathcal{L}_2 + \mathcal{L}_3$ for the integral \eqref{e:1/Gamma}}
 \label{f:hh}

\end{figure}
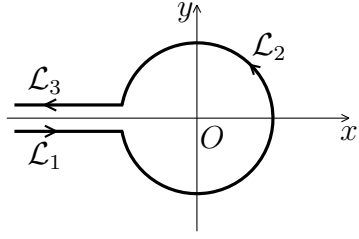

\end{proof}

\subsection{The confluent hypergeometric functions}

Now we summarize some results regarding the confluent hypergeometric functions which are used in this paper. Given $\alpha,\beta\in\dR$ such that $\alpha>\beta$ and $\alpha$ is not a negative integer, we consider the following {\it confluent hypergeometric equation}
\begin{equation}
y\cdot \frac{d^2 \mathcal{J}(y)}{dy^2}+(\fa-y)\cdot \frac{d\mathcal{J}(y)}{dy}-\fb\cdot \mathcal{J}(y)=0.\label{e:c-h-e}\end{equation}
Let 
\begin{equation}
\Ku(\beta,\alpha,y)\equiv\sum\limits_{k=0}^{\infty}\frac{(\beta)_k}{(\alpha)_k}\cdot\frac{y^k}{k!} , \end{equation}
where we define the notation $(x)_k\equiv \prod\limits_{m=1}^{k}(x+m-1)$ and $(x)_0=1$. So the power series $\Ku(\fb,\fa,z)$ is always well-defined for all $\fb\in\dC$, $z\in\dC$ and  $\fa\in\dC\setminus\{0,-1,-2,\ldots\}$. 
Moreover, for any fixed $z\in\dC$, the function $\Ku$ is entire in $\fb$ and meromorphic in $\fa$ with simple poles at negative integers.

It is by straightforward calculations that the function $\Ku(\fb,\fa,y)$ is a solution to \eqref{e:c-h-e}. In the literature, $\Ku$ is called {\it Kummer's (confluent hypergeometric) function}. Moreover, when $y>0$, one can directly check that the function 
$\widehat{\Ku}(\fb,\fa,y)\equiv y^{1-\fa}\cdot \Ku(1+\fb-\fa,2-\fa, y)$, which is linearly independent of $\Ku(\fb,\fa,y)$, also solves \eqref{e:c-h-e}. Therefore, 
the
general solution of \eqref{e:c-h-e} for $y>0$ is
\begin{equation}
\mathcal{J}(y) = C\cdot \Ku(\fb,\fa,y) + C^* \cdot  y^{1-\fa}\cdot \Ku(1+\fb-\fa,2-\fa, y).\label{e:combination}
\end{equation}

The power series definition of $\Ku(\fb,\fa,y)$ immediately gives the following integral representation formula which is well known in the literature. We include a short proof just for the convenience of the readers. 
\begin{lemma}\label{l:Ku-int}
For any $\fa>\fb>0$, then for each $y\in\dR$, \begin{equation}\Ku(\fb,\fa,y) = \frac{\Gamma(\fa)}{\Gamma(\fb)\Gamma(\fa-\fb)}
\int_0^1 e^{yt} t^{\fb-1}(1-t)^{\fa-\fb-1}dt.
\end{equation}

\end{lemma}

\begin{proof}
Given $p,q>0$, let $B(p,q)$ be the beta function which is defined by
\begin{equation}
B(p,q) \equiv \int_0^{1}t^{p-1}(1-t)^{q-1}dt.
\end{equation}
Then the beta function satisfies $B(p,q)=\frac{\Gamma(p)\Gamma(q)}{\Gamma(p+q)}$.
The above formulae imply that
\begin{eqnarray}
\frac{(\fb)_k}{(\fa)_k}
& = & \frac{\Gamma(\fb+k)}{\Gamma(\fb)}\cdot\frac{\Gamma(\fa)}{\Gamma(\fa+k)} \nonumber\\
& =& \frac{\Gamma(\fa)}{\Gamma(\fb)}\cdot \frac{B(\fb+k,\fa-\fb)}{\Gamma(\fa-\fb)}\nonumber\\
&= &\frac{\Gamma(\fa)}{\Gamma(\fb)\Gamma(\fa-\fb)}\int_0^1t^{\fb+k-1}(1-t)^{\fa-\fb-1}dt.
\end{eqnarray}

Now we return to the definition of $\Ku$, combining the above summation, 
\begin{eqnarray}
\Ku(\beta,\alpha,y) &=& \sum\limits_{k=0}^{\infty}\frac{(\beta)_k}{(\alpha)_k}\cdot\frac{y^k}{k!} 
\nonumber\\
&=&\frac{\Gamma(\fa)}{\Gamma(\fb)\Gamma(\fa-\fb)}\int_0^1t^{\beta-1}(1-t)^{\fa-\fb-1}\sum\limits_{k=0}^{\infty}\frac{(yt)^{k-1}}{k!}dt
\nonumber\\
&=&\frac{\Gamma(\fa)}{\Gamma(\fb)\Gamma(\fa-\fb)}\int_0^1e^{yt}t^{\beta-1}(1-t)^{\fa-\fb-1}dt.
\end{eqnarray}
The proof is done.
\end{proof}

Given $\fb>0$ and $y>0$, we define the function
\begin{equation}
\mathcal{U}(\fb,\fa, y) \equiv 
\frac{1}{\Gamma(\fb)}\int_0^{\infty}e^{-yt}t^{\fb-1}(1+t)^{\fa-\fb-1}dt.
\end{equation}
Quick computations show that  for each $\fb>0$, the function 
$\mathcal{U}(\fb,\fa, y) $ is a solution to the confluent hypergeometric equation \eqref{e:c-h-e} on the positive real axis $\dR_+$.
Now let $\fb>0$ and $\fa\in\dR\setminus\{0,-1,-2,-3,\ldots\}$, thanks to \eqref{e:combination}, the function  $\mathcal{U}(\fb,\fa, y)$ can be written in terms of Kummer's function $\Ku$. Evaluating those functions and their derivatives at $y=0$, one can easily obtain
\begin{equation}
\mathcal{U}(\fb,\fa,y) = \frac{\Gamma(1-\fa)}{\Gamma(1+\fb-\fa)}\cdot\Ku(\fb, \fa, y) + 
\frac{\Gamma(\fa-1)}{\Gamma(\fb)}\cdot y^{1-\fa}\cdot \Ku(1+\fb-\fa,2-\fa, y).
\label{e:Ku-Tri}
\end{equation}
Notice that, the above relation is well-defined for each $y\geq 0$ and  non-integral $\alpha$. Moreover, if $\alpha\to n+1\in\dZ_+$, then the right hand side of \eqref{e:Ku-Tri} will tend to a definite limit. 
The function $\mathcal{U}(\fb, \fa ,y )$ is usually called {\it Tricomi's (confluent hypergeometric) function}.
In our context, we are also interested in the case $y<0$. 
It can be directly verified that, if $y<0$, the function
\begin{equation}
\Tri(\fb, \fa , y) \equiv e^{y} \cdot \mathcal{U}(\fa-\fb, \fa, -y)
\end{equation}
solves equation \eqref{e:c-h-e}.
Moreover, it immediately follows from the integral representation of $\mathcal{U}$ that for any $y<0$,
\begin{equation}
\Tri(\beta,\alpha,y)= \frac{e^y}{\Gamma(\fa-\fb)}\int_0^{\infty}e^{yt}t^{\fa-\fb-1}(1+t)^{\fb-1}dt.\end{equation}
In summary, 
if $y<0$, 
the equation \eqref{e:c-h-e} has two linearly independent solutions
$\Ku(\fb,\fa, y)$ and $\Tri(\fb, \fa, y)$.

The asymptotic behavior of $\Ku(\fb,\fa,y)$, $\mathcal{U}(\fb,\fa,y)$ and $\Tri(\fb, \fa , y)$
 can be easily seen from the above integral formulae. In fact, we have the following 
 
\begin{lemma}
\label{l:asymp-Ku-Tri}The following asymptotics hold:
\begin{enumerate}\item
 Let $\fa \in \dR\setminus\{0,-1,-2,-3,\ldots\}$ and $\fb>0$ satisfy $\fa>\fb+1$, then 
 \begin{align}
 \Ku(\fb,\fa,y) \sim 
 \begin{cases}
 \frac{\Gamma(\alpha)}{\Gamma(\alpha-\beta)}\cdot(-y)^{-\beta}, & y\to -\infty, \\
 \frac{\Gamma(\fa)}{\Gamma(\fb)}\cdot e^y\cdot y^{\fb-\fa}, & y\to+\infty.
 \end{cases}
 \end{align}
 \item Let $\beta>0$, then \begin{equation}
\mathcal{U}(\fb,\fa, y) \sim  
y^{-\fb},\ y\to+\infty.
\end{equation}
\item Let $\alpha>\beta$, then 
\begin{equation} \label{eqn-A27}
\Tri(\fb,\fa, y)\sim e^y\cdot(-y)^{\fb-\fa},\ y\to-\infty.
\end{equation}

 \end{enumerate}

\end{lemma}

\begin{proof}
The proof is straightforward.
For example, we only prove \begin{equation}\Ku(\fb,\fa,y) \sim\frac{\Gamma(\alpha)}{\Gamma(\alpha-\beta)}\cdot(-y)^{-\beta}\end{equation} as $y\to -\infty$. The calculations of the remaining cases are the same. We make change of variables and let
$u=-yt$, then 
\begin{align}
\Ku(\fb,\fa,y) &= \frac{\Gamma(\fa)}{\Gamma(\fb)\Gamma(\fa-\fb)}
\int_0^1 e^{yt} t^{\fb-1}(1-t)^{\fa-\fb-1}dt
\nonumber\\
&=\frac{\Gamma(\fa)}{\Gamma(\fb)\Gamma(\fa-\fb)}\cdot (-y)^{-\fb}\cdot
\int_0^{-y}e^{-u}u^{\fb-1}\Big(1+\frac{u}{y}\Big)^{\fa-\fb-1}du.
\end{align}
Since $\fa-\fb-1>0$ and $-1\leq \frac{u}{y}\leq 0$, it is obvious $(1+\frac{u}{y})^{\fa-\fb-1}\leq 1$. Hence dominated convergence theorem implies 
\begin{equation}
\lim\limits_{y\to-\infty}\int_0^{-y}e^{-u}u^{\fb-1}\Big(1+\frac{u}{y}\Big)^{\fa-\fb-1}du=\int_0^{\infty}e^{-u}u^{\fb-1}du = \Gamma(\beta).
\end{equation}
Therefore, as $y\to-\infty$,
\begin{equation}
\Ku(\fb,\fa,y)\sim\frac{\Gamma(\fa)}{\Gamma(\fa-\fb)}\cdot(-y)^{-\fb}.
\end{equation}
\end{proof}

Next we introduce some recurrence formulae for Kummer's function. 

\begin{lemma}Let $\fa \in \dR\setminus\{0,-1,-2,-3,\ldots\}$ and $\fb\in\dR$, then for each $y\in\dR$,
\begin{align}\Ku(\fb,\fa, y) 
&= \Ku(\fb+1,\fa,y) - \frac{y}{\fa}\Ku(\fb+1,\fa+1,y),\label{e:beta-rec}
\\
 \Ku(\fb,\fa,y) &= \frac{\fa+y}{\fa}\cdot\Ku(\fb,\fa+1,y) - \frac{\fa-\fb+1}{\fa(\fa+1)}\cdot y\cdot\Ku(\fb,\fa+1,y).\label{e:alpha-rec}
\end{align}
\end{lemma}

\begin{proof}
The formula can be quickly verified by applying the power series definition of $\Ku$.
\end{proof}

With the above recurrence formula, we can extend the domain of indices in Lemma \ref{l:asymp-Ku-Tri} for Kummer's function.
\begin{lemma}\label{l:general-Ku-asymp}
 For any $\fa \in \dR\setminus\{0,-1,-2,-3,\ldots\}$ and $\beta\in\dR$ such that   
  $\fa>\fb$, then 
 \begin{align} \label{eqn-A34}
 \Ku(\fb,\fa,y) \sim 
 \begin{cases}
 \frac{\Gamma(\alpha)}{\Gamma(\alpha-\beta)}\cdot(-y)^{-\beta}, & y\to -\infty, \\
 \frac{\Gamma(\fa)}{\Gamma(\fb)}\cdot e^y\cdot y^{\fb-\fa}, & y\to+\infty.
 \end{cases}
 \end{align}

\end{lemma}

\begin{proof}
We start with the initial step by assuming $\alpha-\beta> 1$ 
and $\beta>1$. Then Lemma \ref{l:asymp-Ku-Tri} in this case shows that the desired asymptotics hold in this case. 

Applying the recurrence formula \eqref{e:alpha-rec}, we can extend the domain of indices to $\alpha-\beta>0$ and $\beta>1$. Then applying \eqref{e:beta-rec}, one can obtain the desired asymptotics for all $\beta\in\dR$. The proof is done.
\end{proof}

\begin{lemma}
[Kummer's transformation law] \label{l:kummer-transformation} Let $\fa \in \dR\setminus\{0,-1,-2,-3,\ldots\}$ and $\fb\in\dR$, then  
for any $y\in\dR$, 
\begin{equation}\Ku(\fb,\fa,y)=e^y \cdot \Ku(\fa-\fb,\fa,-y).
\end{equation}
\end{lemma}
\begin{proof}

First, we temporarily  assume $\fa>\fb>0$. By Lemma \ref{l:Ku-int}, 
\begin{eqnarray}
e^y\cdot \Ku(\fa-\fb,\fa,-y)
&=& \frac{\Gamma(\fa)}{\Gamma(\fa-\fb)\Gamma(\fb)}
\int_0^1 e^{y(1-t)} t^{\fa-\fb-1}(1-t)^{\fb-1}dt
\nonumber\\
&=& \frac{\Gamma(\fa)}{\Gamma(\fa-\fb)\Gamma(\fb)}
\int_0^1 e^{ys} (1-s)^{\fa-\fb-1}s^{\fb-1}ds
\nonumber\\
&=& \Ku(\fb,\fa, y).
\end{eqnarray}

Now we prove the general case. Since both 
$\frac{e^y\cdot\Ku(\fa-\fb,\fa,-y)}{\Gamma(\fa)}$ and 
$\frac{\Ku(\fb,\fa,y)}{\Gamma(\fa)}$
are entire functions in $\dC$, so the standard analytic continuation theorem implies that 
$\Ku(\fb,\fa,y)=e^y \cdot \Ku(\fa-\fb,\fa,-y)$
holds for any arbitrary $\beta\in\dR$ and $\alpha\in\dR\setminus\{0,-1,-2,-3,\ldots\}$.
\end{proof}

Next we give another integral representation for Kummer's function $\Ku(\fb,\fa,y)$ in the case $y\leq 0$, which has a crucial role in proving the uniform estimates in Section \ref{s:jk not zero}.

\begin{lemma}\label{l:general-Ku-int} Assume that $\fa>\fb$ and $y\leq 0$, then it holds that
\begin{equation}
\Ku(\fb,\fa,y)=\frac{\Gamma(\fa)}{\Gamma(\fa-\fb)}\cdot e^{y}(-y)^{\frac{1-\fa}{2}}\cdot \int_0^{\infty}e^{-t}\cdot t^{\frac{\fa-1}{2}-\fb}\cdot I_{\fa-1}(2\sqrt{-yt})dt.\end{equation}
\end{lemma}

\begin{proof}
By definition,
\begin{equation}
I_{\fa-1}(2\sqrt{-yt})=\sum\limits_{k=0}^{\infty}\frac{(-yt)^{k+\frac{\fa-1}{2}}}{k!\cdot\Gamma(k+\fa)}.
\end{equation}
Integrating the above expansion, it follows that
\begin{eqnarray}
&& \frac{\Gamma(\fa)}{\Gamma(\fa-\fb)}\cdot\int_0^{\infty}e^{-t}\cdot t^{\frac{\fa-1}{2}-\fb}\cdot I_{\fa-1}(2\sqrt{-yt})dt
\nonumber\\
 &=&  \frac{\Gamma(\fa)}{\Gamma(\fa-\fb)}\cdot (-y)^{\frac{\alpha-1}{2}}\cdot\sum\limits_{k=0}^{\infty} \frac{(-y)^k}{k!\cdot\Gamma(k+\fa)}\cdot\int_0^{\infty}e^{-t}\cdot t^{\fa-\fb+k-1}dt
\nonumber\\
&=&   \frac{\Gamma(\fa)}{\Gamma(\fa-\fb)}\cdot (-y)^{\frac{\alpha-1}{2}}\cdot\sum\limits_{k=0}^{\infty}\frac{(-y)^k\cdot \Gamma(\alpha-\beta+k )}{k!\cdot\Gamma(k+\fa)}.
\end{eqnarray}
By the recursive formula of the Gamma function, $\frac{\Gamma(\alpha-\beta+k)}{\Gamma(k+\alpha)} = \frac{(\alpha-\beta)_k \cdot \Gamma(\alpha-\beta)}{(\alpha)_k \cdot \Gamma(\alpha)}$, so it follows that
\begin{eqnarray}
&&\frac{\Gamma(\fa)}{\Gamma(\fa-\fb)}\cdot\int_0^{\infty}e^{-t}\cdot t^{\frac{\fa-1}{2}-\fb}\cdot I_{\fa-1}(2\sqrt{-yt})dt
\nonumber\\
&=& (-y)^{\frac{\alpha-1}{2}}\cdot\sum\limits_{k=0}^{\infty}\frac{(\alpha-\beta)_k (-y)^k}{(\alpha)_k \cdot k!}
\nonumber\\
&=&(-y)^{\frac{\alpha-1}{2}}\cdot\Ku(\fa-\fb,\fa,-y). \end{eqnarray}
Therefore,
\begin{eqnarray}
&&\frac{\Gamma(\fa)}{\Gamma(\fa-\fb)}\cdot e^{y}(-y)^{\frac{1-\fa}{2}}\cdot \int_0^{\infty}e^{-t}\cdot t^{\frac{\fa-1}{2}-\fb}\cdot I_{\fa-1}(2\sqrt{-yt})dt
\nonumber\\
&=& e^y \cdot \Ku(\fa-\fb,\fa,-y)
\nonumber\\
&=&\Ku(\fb,\fa,y).
\end{eqnarray}
The last equality follows from Kummer's transformation law.
\end{proof}

\begin{lemma}\label{l:B-C}
Let $\nu>0$, then for all $y>0$
\begin{align}
I_{\nu}(y) &= \frac{(\frac{y}{2})^{\nu}e^{-y}}{\Gamma(\nu+1)}\Ku(\nu+\frac{1}{2},2\nu+1,2y),\label{e:I-Ku}
\\
K_{\nu}(y) &= \sqrt{\pi}(2y)^{\nu}e^{-y}\mathcal{U}(\nu+\frac{1}{2},2\nu+1,2y).\label{e:K-Tri}
\end{align}

\end{lemma}

\begin{proof}
The relation \eqref{e:I-Ku} can be verified by the power series definition of $I_{\nu}$ and $\Ku(\nu+\frac{1}{2},2\nu+1,2y)$, so we just omit the computations.

To prove \eqref{e:K-Tri}, first we assume $\nu$ is not an integer. Combining the definition 
\begin{equation}
K_{\nu}(y) = \frac{\pi}{\sin(\nu\pi)}\cdot\frac{I_{-\nu}(y) - I_{\nu}(y)}{2}
\end{equation}
and the relation 
\begin{equation}
\mathcal{U}(\nu+\frac{1}{2},2\nu+1,y) = \frac{\Gamma(-2\nu)}{\Gamma(\frac{1}{2}-\nu)}\cdot\Ku(\nu+\frac{1}{2},2\nu+1,y) + 
\frac{\Gamma(2\nu)}{\Gamma(\nu+\frac{1}{2})}\cdot y^{-2\nu}\cdot \Ku(\frac{1}{2}-\nu,1-2\nu, y),
\end{equation}
which is given by \eqref{e:Ku-Tri}.
If $\nu$ is an integer, the relation \eqref{e:K-Tri} can be obtained by the limiting definition of $K_{\nu}$ and the continuity argument for $\nu$. 
\end{proof}

The following corollary shows the asymptotic behavior 
of $I_{\nu}(y)$ and $K_{\nu}(y)$ as $y\to+\infty$.
\begin{corollary}\label{c:bessel-asymp} Let $\nu>0$, then we have
\begin{equation}
\lim\limits_{y\to+\infty}\frac{I_{\nu}(y)}{\frac{e^y}{\sqrt{2\pi y}}}=1\end{equation}
and
\begin{equation}\lim\limits_{y\to+\infty}\frac{K_{\nu}(y)}{\sqrt{\frac{\pi}{2y}}\cdot e^{-y}}=1.
\end{equation}
\end{corollary}
\begin{proof}
The proof follows from Lemma \ref{l:asymp-Ku-Tri}, Lemma \ref{l:general-Ku-asymp} and Lemma \ref{l:B-C}.
\end{proof}

\bibliographystyle{amsalpha} 
\bibliography{References_Liouville}

\end{document}